\documentclass[11pt,a4paper,reqno]{amsart}
\usepackage[left=3cm,right=3cm,top=3cm,bottom=3cm]{geometry}

\usepackage[utf8x]{inputenc}
\usepackage[T1]{fontenc}
\usepackage{mathtools}
\usepackage{amsthm}
\usepackage{amsfonts}
\usepackage{amssymb}
\usepackage{enumitem}
\setlist[enumerate]{label=\emph{(\roman*)}}

\newtheorem{theorem}{Theorem}
\newtheorem{lemma}{Lemma}
\newtheorem{proposition}{Proposition}
\theoremstyle{remark}\newtheorem{remark}{Remark}
\numberwithin{equation}{section}

\DeclareMathOperator{\sech}{sech}
\newcommand\Lp{\mathcal{L}_+}
\newcommand\Lm{\mathcal{L}_-}
\newcommand\Lc{\mathcal{L}_c}
\newcommand\Lu{\mathcal{L}_1}

\newcommand\HH{\mathbf{K}}

\newcommand\MM{\mathbf{M}}
\newcommand\RR{\mathbb{R}}
\newcommand\ZZ{\mathbf{Z}}

\newcommand\EE{\mathcal{E}}

\newcommand\md{\textnormal{M}}
\newcommand\ii{\textnormal{i}}
\newcommand\tin{T_\infty}
\newcommand\tzero{T_0}
\newcommand\tstar{T_\star}
\newcommand\ths{\theta_1}
\newcommand\thc{\theta_{2}}
\newcommand\tht{\theta_{3}}

\newcommand\la{\langle}
\newcommand\ra{\rangle}

\title[Logarithmic 2-solitons for the cubic Schr\"odinger system]{Construction of 2-solitons with logarithmic distance for the one-dimensional cubic Schr\"odinger system}
\author[Y. Martel]{Yvan Martel}
\address{CMLS, \'Ecole Polytechnique, CNRS,  91128 Palaiseau, France}
\email{yvan.martel@polytechnique.edu}
\author[T.V. Nguy\~{\^e}n]{Ti\'{\^e}n Vinh Nguy\~{\^e}n }
\address{CMLS, \'Ecole Polytechnique, CNRS,  91128 Palaiseau, France}
\email{tien-vinh.nguyen@polytechnique.edu}
\begin{document} 
\begin{abstract}
We consider a system of coupled cubic Schr\"odinger equations in one space dimension
\begin{equation*}
\begin{cases}
\ii\partial_t u + \partial_x^2 u +(|u|^2 + \omega |v|^2) u =0\\
\ii\partial_t v + \partial_x^2 v+ (|v|^2 + \omega |u|^2) v=0
\end{cases}\quad
(t,x)\in \RR\times\RR,
\end{equation*}
in the non-integrable case $0 < \omega < 1$.

First, we justify the existence of a symmetric 2-solitary wave with logarithmic distance, more precisely a solution of the system satisfying
\[
\lim_{t\to +\infty}\left\| \begin{pmatrix} u(t) \\ v(t)\end{pmatrix} - \begin{pmatrix} e^{\ii t}Q (\cdot - \frac{1}{2} \log (\Omega t) - \frac{1}{4} \log \log t) \\[4pt] e^{\ii t}Q (\cdot + \frac{1}{2} \log (\Omega t) + \frac{1}{4} \log \log t)\end{pmatrix}\right\|_{H^1\times H^1} = 0\]
where $Q = \sqrt{2}\sech $ is the explicit solution of $ Q'' - Q + Q^3 = 0$ and $\Omega>0$ is a constant.
This result extends to the non-integrable case the existence of symmetric 2-solitons 
with logarithmic distance known in the integrable case $\omega=0$ and $\omega=1$
(\cite{Man,ZS}).
Such strongly interacting symmetric $2$-solitary waves were also previously constructed for the non-integrable scalar nonlinear Schr\"odinger equation in any space dimension and for any energy-subcritical power nonlinearity (\cite{MRlog,NV1}).

Second, under the conditions $0<c<1$ and $0<\omega < \frac 12 c(c+1)$, we construct solutions of the system satisfying
\[
\lim_{t\to +\infty}\left\| \begin{pmatrix}u(t) \\ v(t)\end{pmatrix} - \begin{pmatrix}e^{\ii c^2 t}Q_c (\cdot - \frac{1}{(c+1)c} \log (\Omega_c t) )\\[4pt] e^{\ii t} Q (\cdot + \frac{1}{c+1} \log (\Omega_c t))\end{pmatrix} \right\|_{H^1\times H^1}=0\]
where $Q_c(x)=cQ(cx)$ and $\Omega_c>0$ is a constant.
Such logarithmic regime with non-symmetric solitons does not exist in the integrable cases $\omega=0$ and $\omega=1$
 and is still unknown in the non-integrable scalar case.
\end{abstract}
\maketitle
\section{Introduction}
\subsection{System of cubic Schr\"odinger equations}
We consider the following one dimensional focusing-focusing system of coupled cubic Schr\"odinger equations
\begin{equation}\label{snls}\tag{coupled NLS}
\begin{cases}
\ii\partial_t u + \partial_x^2 u + \left(|u|^2 + \omega |v|^2\right) u =0\\
\ii\partial_t v + \partial_x^2 v + \left(|v|^2 + \omega |u|^2\right) v =0
\end{cases}
 \quad (t,x)\in \RR\times\RR
\end{equation}
for $u(t,x), v(t,x):\RR\times \RR \to \mathbb{C}$ and
for any parameter $0 < \omega < 1$.
The initial data 
$u(0,x) = u_0(x)$, $v(0,x) = v_0(x)$
is taken in $H^1(\RR) \times H^1(\RR) $. The Hamiltonian system \eqref{snls} arises as a model for the propagation of the electrical field in nonlinear optics. Such systems also appear to model the interaction of two Bose-Einstein condensates in different spin states. See \cite{APT,Berge,Ya}.

For $\omega = 0$, the system \eqref{snls} simply reduces to two cubic focusing Schr\"odinger equations without coupling (see \cite{APT,Ya, ZS})
\begin{equation}\label{inls}\tag{cubic NLS}
\ii \partial_t u + \partial_x^2 u + |u|^2 u = 0 \quad (t,x)\in \RR\times\RR.
\end{equation}
For $\omega = 1$, the system \eqref{snls} is called the Manakov system (see \cite{APT,Man,Ya})
\begin{equation}\label{mana}\tag{MS}
\begin{cases}
 \ii\partial_t u + \partial_x^2 u + (|u|^2 + |v|^2) u = 0\\
 \ii\partial_t v + \partial_x^2 v + (|v|^2+ |u|^2 )v = 0.
\end{cases}
\end{equation}
Both \eqref{inls} and \eqref{mana} are completely integrable.
For $0<\omega<1$, the system is not known to be integrable. 

It follows from standard arguments (see \emph{e.g.}~\cite{Ca03, GV}) that the system \eqref{snls} is locally well-posed in $H^1 \times H^1$. In this paper, we work in the framework of such $H^1\times H^1$ solutions.
Moreover, the system is invariant under the following symmetries:
\begin{itemize}
 \item Phase: $\gamma$, $ \gamma' \in \RR,\left( \begin{matrix}
 u_0(x)e^{\ii\gamma}\\
 v_0(x)e^{\ii\gamma'}
 \end{matrix}\right) \mapsto \left( \begin{matrix}
 u(t,x)e^{\ii\gamma}\\
 v(t,x)e^{\ii\gamma'}
 \end{matrix}\right) $;
 \item Scaling: $\lambda > 0$, $ \lambda \left(\begin{matrix}
 u_0 \\ v_0
 \end{matrix} \right)(\lambda x) \mapsto \lambda \left(\begin{matrix}
 u\\ v
 \end{matrix} \right) (\lambda^2 t, \lambda x) $;
 \item Space translation: $\sigma \in \RR$, $ \left( \begin{matrix}
 u_0\\ v_0
 \end{matrix}\right) (x+ \sigma)\mapsto \left(\begin{matrix}
 u\\ v
 \end{matrix} \right) ( t, x+ \sigma)$;
 \item Galilean invariance: $\beta\in \RR$, $ e^{\ii\beta x} \left(\begin{matrix}
 u_0\\ v_0
 \end{matrix} \right) (x) \mapsto e^{\ii\beta(x- \beta t)} \left(\begin{matrix}
 u\\ v
 \end{matrix} \right) (t, x- 2\beta t)$.
\end{itemize}
For $H^1\times H^1$ solutions, the following quantities are constant:
\begin{itemize}
 \item Masses: 
 \begin{equation*}
 M(u(t)) = \int_\RR |u(t,x)|^2 dx = M(u_0), \quad M(v(t)) = \int_\RR |v(t,x)|^2 dx = M(v_0);
 \end{equation*}
 \item Energy:
 \begin{equation*}
 \begin{aligned}
 E(u(t), v(t)) &= \frac{1}{2} \int_\RR \left(|\partial_x u|^2 + |\partial_x v|^2\right) (t,x) dx - \frac{1}{4} \int_\RR \left( |u|^4 + |v|^4 + 2\omega |u|^2|v|^2\right) (t,x) dx \\
 & = E(u_0, v_0);
 \end{aligned}
 \end{equation*}
 \item Momentum: 
 \begin{equation*}
 J(u(t), v(t)) = \Im \int_\RR \partial_x u(t,x) \bar{u} (t,x) dx + \Im \int_\RR \partial_x v(t,x) \bar{v} (t,x) dx = J(u_0, v_0).
 \end{equation*}
\end{itemize}
By the Gagliardo-Nirenberg inequality
$
\|u\|_{L^4}^4 \lesssim \|u\|_{L^2}^3 \|\partial_x u\|_{L^2}$
and standard arguments, the system is globally well-posed in $H^1\times H^1$
 (see \emph{e.g.} \cite{Ca03,We83}).

Let $Q$ be the ground state, defined as
\begin{equation*}
Q(x) = \frac{\sqrt{2}}{\cosh(x)} \mbox{ unique (up to translation) $H^1$ solution of $Q'' - Q + Q^3 = 0$ on $\RR$.}
\end{equation*}
Recall that (cubic NLS) admits solitary wave solutions, also called solitons, of the form
\begin{equation*}
u(t,x) = e^{\ii\gamma + \ii \lambda^2 t + \ii \beta(x- \beta t)} Q_\lambda(x - \sigma - 2 \beta t) \quad \text{with} \quad Q_\lambda(x) = \lambda Q(\lambda x)
\end{equation*}
where $\lambda>0, \gamma,\sigma,\beta\in\RR$. When $v = 0 $ (or $u = 0$), the system \eqref{snls} simplifies into (cubic NLS),
and thus we deduce soliton solutions of \eqref{snls}: 
\[
\begin{pmatrix}
u \\ v
\end{pmatrix}(t,x) = \begin{pmatrix}
e^{\ii \Gamma_1(t,x)} Q_{\lambda_1}(x - \sigma_1 - 2\beta_1 t) \\ 0 
\end{pmatrix},\quad \Gamma_1(t,x)=\gamma_1 + \lambda_1^2 t + \beta_1 (x- \beta_1 t)
\]
and 
\[
\begin{pmatrix}
u \\ v
\end{pmatrix}(t,x) = 
\begin{pmatrix} 0 \\
e^{\ii \Gamma_2(t,x)} Q_{\lambda_2}(x - \sigma_2 - 2 \beta_2 t) 
\end{pmatrix}, \quad \Gamma_2(t,x)=\gamma_2+ \lambda_2^2 t + \beta_2 (x- \beta_2 t)
\] 
for any $\lambda_j > 0, \gamma_j, \sigma_j, \beta_j \in \RR$ $(j = 1,2)$. By definition, a multi-solitary wave (or multi-soliton) is a solution behaving in large time as a sum of such single solitons. In this article, we focus on 2-solitons such that one solitary wave is carried by $u$ and the other one by $v$.

\subsection{Previous results and motivation} Multi-solitons have been studied intensively in the integrable case, \emph{i.e.} for~\eqref{inls} and \eqref{mana}, as well as for some nearly integrable models; see \cite{APT, FT, GO, KS, Olm, Ya, ZS}.
From the inverse scattering theory, there are three types of 2-solitons for \eqref{inls}:
\begin{itemize}
 \item[(a)] Two solitons with different velocities: as $t\to +\infty$, the distance between the solitons is of order $t$ (\cite{ZS}).
 \item[(b)] Double pole solutions: the two solitons have the same amplitude and their distance is logarithmic in $t$ (\cite{Olm, ZS}).
 \item[(c)] Periodic $2$-solitons: the two solitons have different amplitudes and their distance is a periodic function of time (\cite{Ya,ZS}).
\end{itemize}
More generally, the integrability theory treats the case of $K$-solitary waves for any $K\geq 2$.
Moreover, in the integrable case, multi-solitons have a pure soliton behavior for both $t\to +\infty$ and $t \to - \infty $ and describe the elastic interactions between solitons. For \eqref{mana}, a trichotomy similar to (a)-(b)-(c) is studied formally and numerically in \cite{Ya1}.

For non-integrable models, the study of multi-solitons is mostly limited to situations where solitons are decoupled, in particular, asymptotically in large time.
Consider first the scalar nonlinear Schr\"odinger equation 
\begin{equation}\label{nls}\tag{NLS}
\ii \partial_t u + \Delta u + |u|^{p-1} u = 0, \quad u(0,x) = u_0, \quad (t,x)\in \RR\times \RR^d,
\end{equation}
in any space dimension $d\geq 1$
 and for any energy subcritical power nonlinearity 
(\emph{i.e.} $p>1$ for $d=1,2$ and $1<p < 1 + \frac{4}{d-2}$ for $d\geq 3$).
This equation is known to be completely integrable only for $d = 1$ and $p = 3$, \emph{i.e.} \eqref{inls}.
Define the ground state $Q$ as the unique radial positive $H^1$ solution (up to symmetries) of 
$\Delta Q - Q + Q^p = 0$ in $\RR^d$
(for more properties of the ground state, see \cite{Ca03, GSS, PZur, We86})
and $Q_\lambda(x)=\lambda^{\frac 2{p-1}}Q(\lambda x)$, for any $\lambda>0$.
The existence of $K$-solitary waves for \eqref{nls} corresponding to case~(a), \emph{i.e.} solutions $u(t)$ of \eqref{nls} such that
\begin{equation*}
\lim_{t\to +\infty} \bigg\|u(t) - \sum_{k=1}^{K} e^{-\ii\Gamma_k(t,\cdot) } Q_{\lambda_k}(\cdot -\sigma_k- 2 \beta_k t) \bigg\|_{H^1(\RR^d)}=0
\end{equation*}
for any $\lambda_k>0$ and any two-by-two different $\beta_k\in \RR^d$, was established in \cite{CMM, MMnls, Mmulti}.

Recently, the second author proved that the dynamics (b) is also a universal regime for \eqref{nls},
by constructing two symmetric solitary waves with logarithmic distance, \cite{NV1}.
The $L^2$ critical case ($p=1+\frac 4d$), previously studied in~\cite{MRlog}, exhibits a specific blow-up behavior
also related to symmetric $2$-solitons with logarithmic distance in rescaled variables.

Turning back to the system \eqref{snls} in the non-integrable case, \emph{i.e.} for $0<\omega<1$, the existence of multi-solitary wave solutions
corresponding to case (a)
\begin{equation*}
 \lim_{t\to +\infty} \left\|\left( \begin{matrix}
 u(t)\\v(t)
 \end{matrix}\right) - \left(\begin{matrix}
 e^{-\ii\Gamma_1(t,\cdot) }Q_c(\cdot -\sigma_1- 2 \beta_1 t) \\[2pt]
 e^{-\ii\Gamma_2(t,\cdot)} Q(\cdot -\sigma_2- 2\beta_2t )
 \end{matrix} \right) \right\|_{H^1}=0,
 \end{equation*}
 for any $c> 0$ and any different velocities $\beta_1 \neq \beta_2$
 was proved in \cite{DCW} (see also \cite{IL}).

A first goal of this paper is to justify the persistence of the regime (b) for the non-integrable \eqref{snls} in presence of symmetry, following  the articles \cite{MRlog,NV1} for the scalar \eqref{nls} equation.

Second, and more importantly, we investigate the question of the (non-)persistence of the regime~(c).
Indeed, we exhibit a new logarithmic regime corresponding to non-symmetric $2$-solitons with logarithmic distance which replaces the behavior (c). 
At the formal level, the system of parameters of the $2$-solitons is not anymore integrable and periodic solutions disappear, see Remark \ref{book}. A logarithmic regime (see Theorem~\ref{th:2} and Remark \ref{compare1}) then takes place, which does not exist in the integrable cases $\omega=0$ and $\omega=1$.
To our knowledge, such question is open for the scalar equation \eqref{nls} in the non-integrable case
(see Section~\ref{discuss}).
 
\subsection{Main results.} First, we present the symmetric logarithmic regime.
\begin{theorem}
\label{th:1}
For any $0<\omega < 1$, there exists a solution $\left(\begin{smallmatrix} u\\ v\end{smallmatrix}\right)\in \mathcal C(\RR,H^1\times H^1)$ of \eqref{snls} such that
\begin{equation*}
\lim_{t\to +\infty} \left\|\left( \begin{matrix}
u(t)\\ v(t)
\end{matrix}\right) - \left(\begin{matrix}
e^{\ii t} Q(\cdot - \frac{1}{2} \log (\Omega t) - \frac{1}{4} \log \log t ) \\[4pt]
e^{\ii t} Q(\cdot + \frac{1}{2} \log (\Omega t) + \frac{1}{4} \log \log t )
\end{matrix} \right) \right\|_{H^1\times H^1} =0
\end{equation*}
where $\Omega>0$ is a constant depending on $\omega$.\end{theorem}

Note that as $t\to +\infty$, the distance between the two solitary waves is asymptotic to
\begin{equation}\label{dsym}
y(t)= \log t+ \frac{1}{2} \log\log t+\log\Omega.
\end{equation}

\begin{remark}
An analogous dynamics was constructed for \eqref{inls} in \cite{Olm,ZS} 
and for \eqref{nls} in \cite{MRlog,NV1}.
\end{remark}

Second, we construct for \eqref{snls} a new logarithmic dynamics of 2-solitary waves with different amplitude.
\begin{theorem}
\label{th:2}
For any $0 < c <1$ and $0 <\omega < \frac 12 c(c+1) < 1$,
 there exists a solution $\left(\begin{smallmatrix} u\\ v\end{smallmatrix}\right)\in \mathcal C(\RR,H^1\times H^1)$ of \eqref{snls} such that
 \begin{equation*}
 \lim_{t\to +\infty}\left\|\left( \begin{matrix}
 u(t)\\v(t)
 \end{matrix}\right) - \left(\begin{matrix}
 e^{\ii c^2 t}Q_c(\cdot - \frac{1}{(c+1)c} \log (\Omega_c t)) \\[4pt]
 e^{\ii t} Q(\cdot + \frac{1}{c+1} \log (\Omega_c t) )
 \end{matrix} \right) \right\|_{H^1\times H^1} =0 
 \end{equation*}
where $\Omega_c > 0$ is a constant depending on $c$ and $\omega$.\end{theorem}

Note that as $t\to +\infty$, the distance between the two solitary waves is asymptotic to
\begin{equation}\label{dnsym}
y_{c}(t)=\frac{1}{c}\log t +\frac 1c \log \Omega_c \, .
\end{equation}
As mentioned before, such solution does not exist in the integrable cases and 
the analogous question 
for the non-integrable scalar equation \eqref{nls} seems open. See Section~\ref{discuss}.

\begin{remark}\label{compare1}
The slight difference between the two regimes \eqref{dsym} and \eqref{dnsym} is due to stronger interactions when solitary waves have equal amplitudes.
We refer to Sections~\ref{formal1} and~\ref{formalc} for formal derivations of the regimes~\eqref{dsym} and~\eqref{dnsym}.

We believe that there is no other logarithmic regime for \eqref{snls}.
In support of this conjecture, we refer to the case of the generalized Korteweg-de Vries equation, for which existence of a logarithmic regime was proved in \cite{NV2} and uniqueness (in the super-critical case) was established in~\cite{J}.

The case $\frac 12 c(c+1) \leq \omega < 1 $ in Theorem~\ref{th:2} is open
(see step~1 of the proof of Proposition~\ref{prop:coer}).
\end{remark}

\begin{remark}\label{book}
The dynamics of the distance between the two solitary waves is related to nonlinear interactions. A formal study (see notably \cite{GO, KS} and Chapter~4 in \cite{Ya}) shows that the three behaviors (a), (b) and (c) are related to different solutions of
\begin{equation*}
\begin{cases}
\ddot{\gamma} = c_{\gamma} e^{- \sigma}\sin \gamma \\
\ddot{\sigma} = - c_{\sigma} e^{- \sigma} \cos \gamma 
\end{cases}
\end{equation*}
where $\gamma$ is the phase difference, $\sigma$ the relative distance and $c_{\gamma}$, $c_{\sigma}$
are constants. For \eqref{inls}, it holds $c_\gamma = c_\sigma  >0$. Denoting $Y = \sigma + \ii \gamma$, the resulting equation
$\ddot{Y} = - c_\gamma e^Y$
is integrable and admits nontrivial solutions for which $\sigma$ is periodic.
\end{remark}

\begin{remark}
The proofs of Theorems~\ref{th:1} and~\ref{th:2} follow the overall strategy of several previous articles on multi-solitons (\cite{KMR,Martel1,MMnls,MMT1,MMT2,MRlog,Mmulti,NV1,RaSz11}), particularly of \cite{MRlog,NV1} which started the study of multi-solitons with logarithmic distance in a non-integrable setting.
We focus on the proof of Theorem~\ref{th:2}, which is more original in the construction of a suitable approximate solution
and the determination of the asymptotic regime (see Remark~\ref{rk:regime}).

See Section~\ref{discuss} for a comment on the introduction of a refined energy method.
\end{remark}

\subsection{Notation and preliminaries} 
For   complex-valued functions $f,g\in L^2(\RR)$, we denote
\[
\la f,g\ra = \Re \left(\int f \overline g\right).
\]
For $r$ a positive function of time, the notation
$f(t,x) = O_{H^{1}} (r(t))$ means that there exists a constant $C > 0$ such that
$\|f(t)\|_{H^1} \leq C r(t)$.

For any $\lambda>0$ and any function $f$, let
\[
f_\lambda(x) = \lambda f(\lambda x)\quad\mbox{and}\quad
\Lambda f(x)=f(x)+xf'(x)=\partial_\lambda f_\lambda(x)_{|\lambda=1}.
\]
Note the following relation which describes the asymptotics of $Q(x)$ as $x\to-\infty$, 
\begin{equation}\label{asympQ}
Q(x)=\kappa e^{x} -e^{2x} Q(x) \mbox{ on $\RR$ where $\kappa=2\sqrt{2}$.}\end{equation}
Throughout this paper, we consider $\omega$ and $c$ such that
\begin{equation}\label{omega}
0<c\leq 1\quad\mbox{and}\quad 0<\omega<\frac {c(c+1)}2.
\end{equation} 
The linearization of ~\eqref{snls} around solitons involves the following operators:
\begin{equation*}
\Lp =-\partial_x^2+1-3Q^{2},\qquad
\Lm =-\partial_x^2+1-Q^{2},
\quad
\Lc = -\partial_x^2 + c^2 -\omega Q^2.
\end{equation*}
Recall the special relations (\cite{We85}) 
\begin{equation}\label{e:s}
\Lm Q=0,\quad \Lp(\Lambda Q)=-2Q,\quad
\Lp (Q')=0,\quad \Lm(xQ)=- 2 Q'.
\end{equation}
We will use the following properties of these operators.
\begin{lemma}\label{lem:0}
Assume~\eqref{omega}.
\begin{enumerate}
\item There exists $\mu>0$ such that, for all $z\in H^1$, 
\begin{align*} 
&\la \Lp \Re z,\Re z\ra + \la \Lm\Im z,\Im z\ra \geq \mu \|z\|_{H^1}^2
-\frac 1\mu\left( \la z,Q \ra^2+\la z,xQ \ra^2 + \la z, \ii \Lambda Q \ra^2 \right),\\
& \la \Lc z , z \ra \geq \mu \|z\|^2_{H^1}.
\end{align*}
\item 
For any $f \in L^2$, there exists a unique solution $u \in H^2$ of $\Lc u = f$.
Moreover,
\begin{itemize}
\item[--] If $|f(x)| \lesssim e^{-\lambda |x|}$ for some $\lambda >c$, then $|u(x)| \lesssim e^{ - c|x|}$.
\item[--] If $|f(x)| \lesssim e^{-c |x|}$ then $|u(x)| \lesssim (1+|x|)e^{ - c|x|}$.
\end{itemize}
\end{enumerate}
\end{lemma}
\begin{proof}
(i) The coercivity properties of $\Lp$ and $\Lm$ (here in the $L^2$ sub-critical case) are well-known facts (see \emph{e.g.} \cite{MMnls, We85, We86}).

Let $0<\rho<c$ be such that $\omega=\frac12 \rho(\rho+1)$.
By \cite{T} or direct computation, we see that the positive function $Q^\rho$ satisfies $\Lc Q^\rho=(c^2-\rho^2)Q^\rho$.
The coercivity property follows.

(ii) Let $c\leq \lambda\leq 1$. If $\Lc u=f$ with $|f(x)| \lesssim e^{-\lambda |x|}$ then $-u''+c^2 u=g$ where $g=f+\omega Q^2 u$ also satisfies $|g(x)| \lesssim e^{-\lambda |x|}$. The decay properties of $u$ then follows from standard arguments.
\end{proof}
The following result follows directly from Lemma~\ref{lem:0}.
\begin{lemma}\label{le:AB}
\begin{enumerate} 
\item Assume $0<c<1$.
There exists a solution $A$ of 
\begin{equation}\label{eq:A}
\Lc A=-A''+c^2 A-\omega Q^2 A = c\kappa \omega e^{cx} Q^2
\end{equation}
satisfying
\begin{equation}\label{sur:A}
|A(x)|+|A'(x)|+|A''(x)|\lesssim Q_c(x) \quad \mbox{on $\RR$}.
\end{equation}
\item 
There exists a solution $B$ of 
\begin{equation}\label{eq:B}
\Lu B=-B''+ B -\omega Q^2 B = \kappa \omega e^{x} Q^2
\end{equation}
satisfying
\begin{equation}\label{sur:B}
|B(x)|+|B'(x)|+|B''(x)|\lesssim (1+|x|) Q(x) \quad \mbox{on $\RR$}.
\end{equation}
\end{enumerate}
\end{lemma}

\section{Approximate solution in the case $0<c<1$}\label{S:3}
\subsection{Definition of the approximate solution}\label{S:2:1}
Consider $\mathcal{C}^1$ time-dependent real-valued functions 
$\sigma_1$, $\sigma_2$, $\gamma_1$, $\gamma_2$, $\beta_1$, $\beta_2$, to be fixed later
and set
\[\sigma = \sigma_1 - \sigma_2 , \quad \beta = \beta_1 - \beta_2 , \quad \gamma = \gamma_1 - \gamma_2.\]
Denote
\begin{align*}
U&=P+\varphi,\quad P (t,x) = Q_c(x - \sigma_1 (t)) e^{\ii\Gamma_1(t,x)}, \quad \varphi(t,x) = e^{-c\sigma(t)} A(x-\sigma_2(t)) e^{\ii \Gamma_1(t,x)},
\\
V&=R,\quad R (t,x) = Q(x - \sigma_2 (t)) e^{\ii\Gamma_2(t,x)},
\end{align*}
where
\[\Gamma_1 (t,x) = c^2 t + \gamma_1 (t)+ \beta_1(t) x,\quad
\Gamma_2 (t,x) = t + \gamma_2 (t)+ \beta_2(t) x.\]
Introduce the notation
\begin{align*}
&\partial_1P=Q_c'(x-\sigma_1)e^{\ii \Gamma_1},\quad
x_1P=(x-\sigma_1)P,\quad \Lambda_1 P = \Lambda Q_c (x-\sigma_1)e^{\ii \Gamma_1},\\
&\partial_1\varphi=e^{-c\sigma}A'(x-\sigma_2)e^{\ii \Gamma_1},\quad
x_2\varphi=(x-\sigma_2)\varphi,\\
&\partial_2R=Q'(x-\sigma_2)e^{\ii \Gamma_2},\quad
x_2R=(x-\sigma_2)R,\quad \Lambda_2 R = \Lambda Q (x-\sigma_2)e^{\ii \Gamma_2}.
\end{align*}
Define the approximate solution
\begin{equation*}
\ZZ = \begin{pmatrix}
U \\ V 
\end{pmatrix}\quad \mbox{and set}\quad
\EE_\ZZ = \begin{pmatrix}
\EE_{U}\\ \EE_{V}
\end{pmatrix} = \begin{pmatrix}
\ii\partial_t U + \partial_x^2 U + \left(|U|^2 + \omega |V|^2\right) U\\[4pt]
\ii\partial_t V + \partial_x^2 V + \left(|V|^2 + \omega |U|^2\right) V
\end{pmatrix}.
\end{equation*}

\begin{lemma}
It holds
\begin{equation}\label{on:UV}\left\{\begin{aligned}
\EE_{U}&= F-\vec{m}_1 \cdot \vec\md_1-\vec{m}_\varphi\cdot\vec\md_\varphi\\
\EE_{V}&= G-\vec{m}_2 \cdot \vec\md_2
\end{aligned}\right.\end{equation}
where
\begin{equation}\label{def:FG}\left\{\begin{aligned}
F&=3|P|^2\varphi+3|\varphi|^2P+|\varphi|^2\varphi-\omega e^{2c (x -\sigma_1)} |R|^2 P\\
G&=\omega |P+\varphi|^2R 
\end{aligned}\right.\end{equation}
and
\begin{align*}
\vec{m}_1 =
\begin{pmatrix} \dot{\sigma}_1-2\beta_1 \\ \dot\gamma_1 +\dot\beta_1 \sigma_1 +\beta_1^2\\ \dot{\beta}_1 \end{pmatrix},
&\quad 
\vec\md_1=\begin{pmatrix}\ii \partial_1P \\ P \\x_1P \end{pmatrix}
\\
\vec{m}_\varphi =
\begin{pmatrix}\dot{\sigma}_2-2\beta_1 \\ \dot\gamma_1 +\dot\beta_1 \sigma_2 +\beta_1^2+\ii c \dot\sigma\\ \dot{\beta}_1 \end{pmatrix},
&\quad
\vec\md_\varphi=\begin{pmatrix}\ii \partial_1 \varphi \\ \varphi \\ x_2\varphi \end{pmatrix}
\\
\vec{m}_2 =
\begin{pmatrix}\dot{\sigma}_2-2\beta_2 \\ \dot\gamma_2 +\dot\beta_2 \sigma_2 +\beta_2^2\\ \dot{\beta}_2 \end{pmatrix},&\quad
\vec\md_2=\begin{pmatrix} \ii \partial_2 R\\ R \\ x_2R \end{pmatrix}.
\end{align*}
\end{lemma}
\begin{proof}
Using $Q_c''-c^2 Q_c=Q_c^3$ and \eqref{eq:A}, we compute
\begin{align*}
&\ii \partial_t P + \partial_x^2 P+|P|^2P=
-\vec{m}_1\cdot\vec\md_1,\\
&
\ii \partial_t \varphi + \partial_x^2 \varphi +\omega |R|^2 (P+\varphi)
=-\vec{m}_\varphi\cdot\vec\md_\varphi+\omega |R|^2 \left[Q_c(x-\sigma_1)-c\kappa e^{c(x-\sigma_1)}\right]e^{\ii\Gamma_1}.
\end{align*}
Using \eqref{asympQ}, we obtain \eqref{on:UV} for $\EE_U$ with $F$ defined as in \eqref{def:FG}.

Similarly, the equation
\[
\ii \partial_t R + \partial_x^2 R+|R|^2R=-\vec{m}_2\cdot\vec\md_2
\]
implies \eqref{on:UV} for $\EE_V$ with $G$ defined as in \eqref{def:FG}.
\end{proof}

\subsection{Projection of the error terms}
The soliton dynamics is expected to be determined by the following projections
\[
a= \frac 1{2c}{\la F,\partial_1 P\ra} \quad\mbox{and}\quad
b= \frac 12 {\la G,\partial_2 R\ra}.\]
Using $\la \partial_1 P,x_1P\ra=\la Q_c', xQ_c\ra=-\frac 12 \|Q_c\|_{L^2}^2=-2c$ and $\la \partial_2 R,x_2R\ra=-\frac 12 \|Q\|_{L^2}^2=-2$,
we decompose $F$ and $G$ as follows
\begin{equation}\label{onFFGG}\left\{\begin{aligned}
F&=F^\perp - a x_1 P ,\quad \la F^\perp, \partial_1P \ra= 0 \\
G&=G^\perp - b x_2 R,\quad \la G^\perp,\partial_2 R\ra = 0
\end{aligned}\right.\end{equation}
so that \eqref{on:UV} rewrites
\begin{equation}\label{new:EF}\left\{\begin{aligned}
\EE_{U}&= F^\perp-\vec{m}_1^a \cdot \vec\md_1-\vec{m}_\varphi\cdot\vec\md_\varphi\\
\EE_{V}&= G^\perp-\vec{m}_2^b \cdot \vec\md_2
\end{aligned}\right.\end{equation}
with
\begin{equation*}
\vec{m}_1^a =
\begin{pmatrix} \dot{\sigma}_1-2\beta_1 \\ \dot\gamma_1 +\dot\beta_1 \sigma_1 +\beta_1^2\\ \dot{\beta}_1 +a \end{pmatrix}
\quad\mbox{and}\quad
\quad 
\vec{m}_2^b =
\begin{pmatrix} \dot{\sigma}_2-2\beta_2 \\\dot\gamma_2 +\dot\beta_2 \sigma_2 +\beta_2^2\\ \dot{\beta}_2 +b\end{pmatrix}.
\end{equation*}
We compute the main order of these projections.
\begin{lemma}
Let $1<\theta< \min \big\{\frac 1c ; 2 \big\}$.
It holds
\begin{equation}\label{est:ab}
a= \alpha_c e^{-2c\sigma} +O(e^{-2 c\theta \sigma}),\quad
b= - c\alpha_c e^{-2c\sigma} +O(e^{-2 c\theta \sigma})
\end{equation}
where
\begin{equation*}
\alpha_c = 4c^2 \omega \|e^{cx} Q\|_{L^2}^2+ \frac {1}{2} \la \Lc A,A\ra>0.
\end{equation*}
\end{lemma}
\begin{remark}\label{rk:regime}
The   expression of the positive constant $\alpha_c$, relevant in the dynamics of the $2$-soliton (see Section~\ref{formalc}),
suggests that even at the formal level, the introduction of the approximate solution $\big(\begin{smallmatrix}U\\ V\end{smallmatrix}\big)$
including the refined term $\varphi$ is necessary to determine correctly the non-symmetric logarithmic regime.
\end{remark}
\begin{proof}
We start by proving the following estimates
\begin{align}
&\int e^{2c(x-\sigma)}Q_c^2 (x-\sigma) Q^2(x) dx =O(e^{-2c\theta \sigma}),\label{on:QQc2}\\
&\int Q_c^2(x-\sigma) Q(x)Q'(x) dx = -c^3\kappa^2 e^{-2c\sigma} \int e^{2cx}Q^2(x)dx + O(e^{-2c\theta \sigma}).\label{on:QQc}
\end{align}
Proof of \eqref{on:QQc2}. By \eqref{asympQ} and the condition on $\theta$, we have
\[
e^{2c(x-\sigma)}Q_c^2 (x-\sigma) Q^2(x) \lesssim e^{2c\theta(x-\sigma)} Q^2(x)
\lesssim e^{-2c\theta \sigma} e^{-2(1-c\theta) |x|},
\]
and \eqref{on:QQc} follows.

Proof of \eqref{on:QQc}.
It follows from \eqref{asympQ} that
\[
Q_c^2(x) = c^2\kappa^2 e^{2cx} + O(e^{3cx} Q_c(x)),
\]
and so
\begin{equation*}
Q_c^2(x-\sigma) = c^2\kappa^2 e^{-2c\sigma} e^{2cx} + O(e^{-2c\theta \sigma}e^{2c\theta x}).
\end{equation*}
Thus
\[
\int Q_c^2(x-\sigma) Q(x)Q'(x) dx =
c^2\kappa^2 e^{-2c\sigma} \int e^{2cx}Q(x)Q'(x)dx + O(e^{-2c\theta \sigma}).
\]
and \eqref{on:QQc} follows by integration by parts.

\smallskip

From the expression of $F$ in \eqref{def:FG}, we have
\begin{align*}
\la F,\partial_1 P\ra 
& =3e^{-c\sigma} \int Q_c^2(x)Q_c'(x) A(x+\sigma) dx 
+3e^{-2c\sigma} \int Q_c(x)Q_c'(x) A^2(x+\sigma)dx\\
& \quad +e^{-3c\sigma} \int Q_c'(x)A^3(x+\sigma) dx
- \omega \int e^{2cx}Q_c (x)Q_c'(x)Q^2(x+\sigma) dx.
\end{align*}
For the first term, using $-(Q_c')''+c^2 Q_c'=3Q_c^2Q_c'$ (obtained by differentiating the equation of $Q_c$) and the equation $A$ in \eqref{eq:A}, we compute
\begin{align*}
3\int Q_c^2(x)Q_c'(x) A(x+\sigma) dx
&= \int Q_c'(x-\sigma) (-A''(x)+c^2 A(x)) dx\\
&=\omega \int Q_c'(x-\sigma) \left[ Q^2(x) A(x)+c\kappa e^{cx} Q^2(x)\right] dx.
\end{align*}
Similarly as in the proof of \eqref{on:QQc}, using \eqref{asympQ} we observe
\begin{align*}\int Q_c'(x-\sigma) Q^2(x) A(x) dx& = 
 c^2 \kappa e^{-c\sigma} \int e^{cx} Q^2(x) A(x) dx + O(e^{-c\theta \sigma}), \\
 \int Q_c'(x-\sigma) e^{cx} Q^2(x)dx&=c^2\kappa e^{-c\sigma} \int e^{2cx} Q^2(x)dx+ O(e^{-c\theta \sigma}).
\end{align*}
Moreover, it follows from \eqref{eq:A} and the coercivity of the operator $\Lc$ that
\[
 c \kappa\omega \int e^{cx} Q^2(x) A(x) dx
 =\la \Lc A,A\ra >0.
\]
Last, we check using the decay property of $A$ in \eqref{sur:A} and the condition on $\theta$ that
\begin{equation*}
\int Q_c(x)Q_c'(x) A^2(x+\sigma)dx =O(e^{-c \theta\sigma}),\quad
\int Q_c'(x)A^3(x+\sigma) dx =O(e^{-c \sigma}).
\end{equation*}
Using also \eqref{on:QQc2} and $\kappa^2=8$, we find
\begin{equation*}
a=\frac {e^{-2c\sigma}}{2} \left[ c^2\kappa^2 \omega \int e^{2cx} Q^2(x)dx+ \la \Lc A,A\ra\right]
+O(e^{-2c\theta \sigma})=\alpha_c e^{-2c\sigma} +O(e^{-2 c\theta \sigma}).
\end{equation*}

From the definition of $G$, we have
\begin{align*}
\la G,\partial_2 R\ra
&=\omega \int Q_c^2(x-\sigma) Q(x)Q'(x)dx
\\&\quad+2\omega e^{-c\sigma} \int Q_c(x-\sigma)A(x) Q(x)Q'(x) dx+\omega e^{-2c\sigma} \int A^2(x)Q(x)Q'(x)dx.
\end{align*}
On the one hand, integrating by parts, it holds
\[
\la \Lc A,A'\ra = -\omega \int Q^2(x)A(x)A'(x) dx
=\omega \int A^2(x)Q(x)Q'(x)dx.
\]
On the other hand, using \eqref{eq:A} and then integration by parts , it holds
\begin{align*}
\la \Lc A,A'\ra
&= c \kappa\omega\int e^{cx} Q^2(x) A'(x) dx\\
&=-c^2 \kappa\omega \int e^{cx} Q^2(x) A(x) dx 
-2c \kappa\omega \int e^{cx} Q(x)Q'(x) A(x) dx
\\
&=-c \la \Lc A,A\ra
-2c \kappa\omega \int e^{cx} Q(x)Q'(x) A(x) dx.
\end{align*}
Thus,
also using 
\[
\int Q_c(x-\sigma)A(x) Q(x)Q'(x) dx = c\kappa e^{-c\sigma} \int e^{cx} Q(x)Q'(x)A(x) Q dx+O(e^{-c\theta\sigma})
\]
and \eqref{on:QQc}, we obtain
$b=-c \alpha_c e^{-2c\sigma} +O(e^{-2 c\theta \sigma})$.
\end{proof}
\subsection{Formal discussion}\label{formalc}
Formally, the previous computations lead us to the system
\[
\dot\sigma_1=2\beta_1,\quad \dot\beta_1=- \alpha_c e^{-2c\sigma},\quad
\dot\sigma_2=2\beta_2,\quad \dot\beta_2= c\alpha_c e^{-2c\sigma}.
\]
Recalling $\sigma=\sigma_1-\sigma_2$ and $\beta=\beta_1-\beta_2$, this gives
\[
\ddot \sigma = - 2(c+1) \alpha_c e^{-2c\sigma},\quad 2 \beta= {\dot\sigma} ,
\]
which admits the following solution
\[
\sigma(t)=\frac 1c\log (\Omega_c t),\quad 
2\beta(t)=\frac 1{ct}=\frac{\Omega_c}c e^{-c\sigma} \quad \mbox{where}\quad
\Omega_c = [ 2c(c+1) \alpha_c ]^{\frac 12}>0.\]
This justifies the existence of the regime \eqref{dnsym} of Theorem~\ref{th:2}.
In particular, observe that the positive sign of the constant~$\alpha_c$ is responsible for the emergence of the special non-symmetric logarithmic
regime.
The phase parameters $\gamma_1$ and $\gamma_2$ are not essential for the dynamics and so we do not discuss them here.
\subsection{Decomposition around the approximate solution}
Let $\tin \gg1$ to be fixed later and consider a solution $\big(\begin{smallmatrix}u\\ v\end{smallmatrix}\big)$
of \eqref{snls} under the form
\begin{equation}\label{decomposition}
\begin{pmatrix} u\\ v\end{pmatrix}
=\begin{pmatrix} U\\ V\end{pmatrix}+\begin{pmatrix} \varepsilon\\ \eta\end{pmatrix}
\quad \mbox{with}\quad 
\begin{pmatrix} \varepsilon \\ \eta\end{pmatrix}(\tin)=\begin{pmatrix}0\\0\end{pmatrix}.
\end{equation}
Then, using the notation
\[
h(u,v)=\left(|u|^2+\omega|v|^2\right) u
\]
the function $\big(\begin{smallmatrix}\varepsilon\\ \eta\end{smallmatrix}\big)$ satisfies the system
\begin{equation}\label{eq:ee}
\left\{\begin{aligned}
&\ii \partial_t\varepsilon+\partial_x^2 \varepsilon
+h(U+\varepsilon,V+\eta)-h(U,V)+\EE_U=0\\
&\ii \partial_t\eta+\partial_x^2 \eta
+h(V+\eta,U+\varepsilon)-h(V,U)+\EE_V=0
\end{aligned}\right.
\end{equation}
The parameters $\sigma_1$, $\sigma_2$, $\gamma_1$, $\gamma_2$, $\beta_1$ and $\beta_2$ in the definition of $\big(\begin{smallmatrix}U\\ V\end{smallmatrix}\big)$ are fixed by imposing the following orthogonality conditions
\begin{equation}\label{ortho}\left\{\begin{aligned}
&\la \varepsilon, x_1 P\ra =\la \varepsilon, \ii \Lambda_1 P\ra =\la \varepsilon,\ii \partial_1 P\ra=0\\
&\la \eta , x_2 R\ra =\la \eta, \ii \Lambda_2 R \ra = \la \eta,\ii \partial_2 R\ra=0
\end{aligned}\right.\end{equation}
and initial conditions
\begin{equation}\label{initial}\left\{\begin{aligned}
&\sigma_1(\tin)=\frac{\sigma_\infty}{c+1} ,\quad \sigma_2(\tin)=-\frac{c\sigma_\infty}{c+1},\\
&\beta_1(\tin)=\frac{\beta_\infty}{c+1},\quad \beta_2(\tin)=-\frac{c\beta_\infty}{c+1},\\
&\gamma_1(\tin)=0,\quad \gamma_2(\tin)=0,
\end{aligned}\right.\end{equation}
where $\sigma_\infty$ is to be chosen later close to $\frac 1c\log (\Omega_c \tin)$ (see below \eqref{eq:siginf}) and 
\begin{equation}\label{eq:betainf}
\beta_\infty= \frac{\Omega_c}{2c} e^{-c\sigma_\infty}.
\end{equation}
Indeed, by a standard argument and the initial conditions (including $\varepsilon(\tin)=\eta(\tin)=0$), the orthogonality conditions are equivalent to a first order differential system in the parameters $(\sigma_1,\sigma_2,\gamma_1,\gamma_2,\beta_1,\beta_2)$, which admits a unique local solution in the regime considered in this paper.
See \emph{e.g.} Lemma~2.7 in ~\cite{CM} for a detailled argument in the case of the (gKdV) equation, and Lemma~\ref{le:mod} in the present paper for the corresponding estimates on the time derivatives of the parameters.
For technical reasons, one can fix zero initial conditions on $\gamma_1$, $\gamma_2$ as
in \eqref{initial}, but the initial conditions on $\sigma_1$, $\sigma_2$, $\beta_1$ and $\beta_2$ have to depend on a parameter
$\sigma_\infty$ to be fixed later by a topological argument.

As in \cite{MRlog,NV1,RaSz11}, the orthogonality conditions in \eqref{ortho} are related to \eqref{e:s}.
Using the conservation of masses and $L^2$ sub-criticality, we avoid the modulation of the scaling parameters of the solitons
(see \cite{We86} and the proof of Lemma~\ref{le:mod}).

\section{Proof of Theorem~\ref{th:2}}\label{sec:th:2}
\subsection{Bootstrap bounds}
Fix $\ths$, $\thc$ and $\tht$ such that 
$1<\tht<\thc<\ths<\min\big\{\frac 1c;2\big\}$.
Following Section~\ref{formalc}, we work under the following bootstrap estimates, for $1\ll t\leq \tin$,
\begin{equation}\label{BS}\left\{\begin{aligned}
&\|\varepsilon\|_{H^1}+\|\eta\|_{H^1}\leq t^{-\ths},\\
& \bigg|\beta-\frac 1{2ct}\bigg|+\bigg|\beta_1-\frac1{2c(c+1)t}\bigg|+\bigg|\beta_2+\frac{1}{2(c+1)t}\bigg|\leq t^{-\tht},\\
& \bigg|\sigma_1-\frac{\log(\Omega_c t)}{c(c+1)}\bigg|+\bigg|\sigma_2+\frac{\log(\Omega_c t)}{c+1}\bigg|\leq t^{1-\tht},\\
& \bigg| \frac{e^{c\sigma}}{\Omega_c t}-1 \bigg|\leq t^{1-\thc}.
\end{aligned}\right.\end{equation}
For consistency, the free parameter $\sigma_\infty$ in \eqref{initial} will have to be chosen such that
\begin{equation}\label{eq:siginf}
\bigg| \frac{e^{c\sigma_\infty}}{\Omega_c \tin}-1 \bigg|\leq \tin^{1-\thc}.
\end{equation}
\begin{lemma}\label{tech}
Let $0<c_1\leq c_2$ and $q\geq 0$. It holds, for $\sigma> 1$,
\[
\int (1+|x-\sigma|)^q e^{-c_1|x-\sigma|}e^{-c_2|x|} dx\lesssim
\begin{cases}
\sigma^{q+1} e^{-c_1 \sigma} & \mbox{if $c_1=c_2$}\\
\sigma^q e^{-c_1 \sigma} & \mbox{if $c_1\neq c_2$}.
\end{cases}
\]
\end{lemma}
\begin{proof}
We decompose
\begin{multline*}
 \int (1+|x-\sigma|)^q e^{-c_1|x-\sigma|}e^{-c_2|x|} dx
=e^{-c_1\sigma} \int_{-\infty}^{0} (1+|x-\sigma|)^q e^{(c_1+c_2)x}dx \\
 + e^{-c_1\sigma} \int_0^{\sigma} (1+|x-\sigma|)^q e^{-(c_2-c_1)x}dx
+e^{c_1\sigma} \int_{\sigma}^{+\infty} (1+|x-\sigma|)^q e^{-(c_1+c_2)x}dx.
\end{multline*}
The result follows by integration.
\end{proof}

\begin{lemma}
The following hold
\begin{equation}\label{dtP}\begin{aligned}
\|\partial_t P-\ii c^2P\|_{L^2}
&\lesssim \left(|\dot \gamma_1|+|\dot\beta_1||\sigma_1|+|\dot \sigma_1|\right),
\\
\|\partial_t\varphi-\ii c^2\varphi\|_{L^2}
&\lesssim \left(|\dot \gamma_1|+|\dot\beta_1||\sigma_2|+|\dot \sigma_2|+|\dot \sigma|\right)e^{-c\sigma},\\
\|\partial_t R-\ii R\|_{L^2}
&\lesssim \left(|\dot \gamma_2|+|\dot\beta_2||\sigma_2|+|\dot \sigma_2|\right).
\end{aligned}\end{equation}
Let $1<\theta< \min \big\{\frac 1c ; 2 \big\}$. The following hold
\begin{align}
&\|F\|_{L^2}+\|F^\perp\|_{L^2}\lesssim e^{-2c\sigma},\label{F}\\
&\|\partial_t F - \ii c^2 F\|_{L^2}+\|\partial_t F^\perp - \ii c^2 F^\perp\|_{L^2}\lesssim
\left(|\dot \gamma_1|+|\dot\beta_1||\sigma_1|+|\dot\sigma_1|+|\dot\sigma|\right)e^{-2c\sigma},\label{dtF}\\
&\|G\|_{L^2}+\|G^\perp\|_{L^2}\lesssim e^{-c\theta\sigma},\label{G}\\
&\|\partial_t G - \ii G\|_{L^2}+\|\partial_t G^\perp - \ii G^\perp\|_{L^2}\lesssim
\left(|\dot \gamma_2|+|\dot\beta_2||\sigma_2|+|\dot\sigma_2|+|\dot\sigma|\right)e^{-c\theta\sigma}\label{dtG}.
\end{align}
\end{lemma}
\begin{proof} 
Estimates \eqref{dtP} are simple consequences of the definitions of $P$, $\varphi$ and $R$.

Proof of \eqref{F}.
Recall that $F(t,x)=F_1(t,x-\sigma_1(t))e^{\ii \Gamma_1(t,x)}$, where
\[
F_1=3e^{-c\sigma}Q_c^2A(x+\sigma)+3e^{-2c\sigma}Q_cA^2(x+\sigma)+e^{-3c\sigma}A^3(x+\sigma)-\omega e^{2cx}Q^2(x+\sigma)Q_c.
\]
Moreover, from \eqref{sur:A} and Lemma~\ref{tech}, it holds
\[
\|Q_c^2A(x+\sigma)\|_{L^2}+
\|Q_cA^2(x+\sigma)\|_{L^2}\lesssim e^{-c\sigma},
\]
and $\|e^{2cx}Q^2(x+\sigma)Q_c\|_{L^2}\lesssim e^{-2c\sigma}\|e^{2cx}Q^2\|_{L^2}\lesssim e^{-2c\sigma}$.

Proof of \eqref{dtF}. Note that 
\[
\partial_t F- \ii c^2 F= \ii (\dot\gamma_1+\dot \beta_1 \sigma_1) F+ \ii \dot \beta_1 (x-\sigma_1) F 
-\dot\sigma_1 \partial_x F_1 (t,x-\sigma_1)e^{\ii \Gamma_1}
+\partial_t F_1(t,x-\sigma_1)e^{\ii \Gamma_1}.
\]
We see from the expression of $F_1$ and similar estimates that the following hold
\begin{gather*}
\|(\dot\gamma_1+\dot\beta_1 \sigma_1) F_1\|_{L^2}\lesssim (|\dot\gamma_1|+|\dot \beta_1||\sigma_1| )e^{-2c\sigma},\quad
\| xF_1 \|_{L^2} \lesssim e^{-2c\sigma},
\\
\|\partial_x F_1\|_{L^2}\lesssim e^{-2c\sigma},\quad 
\|\partial_t F_1\|_{L^2} \lesssim |\dot \sigma|e^{-2c\sigma}.
\end{gather*}
This proves estimate~\eqref{dtF} for $F$.

Next, note that from the definition of $a$, we have
\[\dot a=\frac 1{2c} \la \partial_t F-\ii c^2 F,\partial_1 P\ra
+\frac 1{2c} \la F,\partial_t \partial_1 P -\ii c^2 \partial_1 P\ra.\]
Thus, from the analogue of \eqref{dtP} for $\partial_1 P$ and \eqref{F}-\eqref{dtF}, we deduce 
\[
|\dot a|\lesssim \left(|\dot \gamma_1|+|\dot\beta_1||\sigma_1|+|\dot\sigma_1|+|\dot\sigma|\right)e^{-2c\sigma}.
\]
Estimate \eqref{dtF} for $F^\perp$ then comes from
\[
\partial_t F^\perp-\ii c^2 F^\perp
= \partial_t F-\ii c^2 F +\dot ax_1 P + a \left[\partial_t (x_1P)-\ii c^2 (x_1P)\right]
\]
and the analogue of \eqref{dtP} for $x_1P$.

Proof of \eqref{G}. We rewrite
$G(t,x)=G_2(x-\sigma_2(t)) e^{\ii \Gamma_2(t,x)}$, where
\[
G_2=\omega Q_c^2(x-\sigma)Q+2\omega e^{-c\sigma} Q_c(x-\sigma)AQ + \omega e^{-2c\sigma}A^2Q.
\]
From Lemma~\ref{tech} and the definition of $\theta$, we have
\[
\|Q_c^2(x-\sigma)Q\|_{L^2}\lesssim e^{-c\theta\sigma},\quad
 \|Q_c(x-\sigma)AQ\|_{L^2}\lesssim e^{-c\sigma}.
\]

Proof of \eqref{dtG}. We have
\[
\partial_t G- \ii G= \ii (\dot\gamma_2+\dot \beta_2 \sigma_2) G+ \dot \beta_2 (x-\sigma_2) G
-\dot\sigma_2 \partial_x G_2 (t,x-\sigma_2)e^{\ii \Gamma_2}
+\partial_t G_2(t,x-\sigma_2)e^{\ii \Gamma_2}.
\]
As before, we use the following estimates to prove \eqref{dtG} for $G$
\begin{gather*}
\|(\dot\gamma_2+\dot\beta_2 \sigma_2) G_2\|_{L^2}\lesssim (|\dot\gamma_2|+|\dot \beta_2||\sigma_2| )e^{-c\theta\sigma},\quad
\| xG_2 \|_{L^2} \lesssim e^{-c\theta\sigma},
\\
\|\partial_x G_2\|_{L^2}\lesssim e^{-c\theta\sigma},\quad 
\|\partial_t G_2\|_{L^2} \lesssim |\dot \sigma|e^{-c\theta\sigma}.
\end{gather*}
The proof of \eqref{dtG} for $G^\perp$ follows from similar arguments and it is omitted.
\end{proof}

\subsection{Modulation equations}
\begin{lemma}\label{le:mod}
Let $\ths<\theta< \min \big\{\frac 1c ; 2 \big\}$.
It holds
\begin{align}
&|\la \varepsilon,P\ra|\lesssim t^{-2}\log t,\quad |\la \eta,R\ra|\lesssim t^{-2\ths}, \label{ePetaR}\\
&|\dot \sigma_1-2\beta_1|+|\dot \sigma_2-2\beta_2|+|\dot\gamma_1|+|\dot\gamma_2|\lesssim t^{-\theta}. \label{param}\\
&|\vec{m}_1|+|\vec{m}_2|+|\vec{m}_1^a|+|\vec{m}_2^b|\lesssim t^{-\theta},\quad
|\vec{m}_\varphi|\lesssim t^{-1} \label{m},\\
&|\dot \beta_1 +a|+|\dot \beta_2+b|\lesssim t^{-1-\ths}.\label{beta}
\end{align}
\end{lemma}
\begin{proof}
Proof of \eqref{ePetaR}.
First, it follows from Lemma~\ref{tech} and \eqref{BS} that
\[
\|U\|_{L^2}^2=\|Q_c+e^{-c\sigma} A(\cdot +\sigma)\|_{L^2}^2 = \|Q_c\|_{L^2}^2+O(t^{-2}\log t).
\]
We use the mass conservation for $u$ and $\varepsilon(\tin)=0$,
\[
\|U+\varepsilon\|_{L^2}^2=\|u\|_{L^2}^2=\|u(\tin)\|_{L^2}^2=\|U(\tin)\|_{L^2}^2=\|Q_c\|_{L^2}^2+O(\tin^{-2}\log \tin),
\]
and thus by \eqref{BS},
\[
2\la \varepsilon,U\ra = \|U+\varepsilon\|_{L^2}^2-\|U\|_{L^2}^2-\|\varepsilon\|_{L^2}^2
=O(t^{-2}\log t).
\]
Last, using $|\la \varepsilon,\varphi\ra|\leq \|\varepsilon\|_{L^2}\|\varphi\|_{L^2}
\lesssim t^{-1} \|\varepsilon\|_{L^2}\lesssim t^{-1-\ths}$
and $2\la \varepsilon,P\ra =2\la \varepsilon,U\ra - 2\la \varepsilon,\varphi\ra$, we obtain
$|\la \varepsilon,P\ra|\lesssim t^{-2}\log t$. The estimate on $\la \eta,R\ra$ follows directly from
$\|v\|_{L^2}=\|v(\tin)\|_{L^2}$.

\smallskip

Proof of \eqref{param}-\eqref{m}-\eqref{beta}. We use the special choice of orthogonality
conditions \eqref{ortho} as well as the relations \eqref{e:s}. We refer to the proof of Lemma~7 in \cite{MRlog}
for a similar argument.
First, differentiating the second orthogonality in \eqref{ortho} and using \eqref{eq:ee},
\begin{align*}
0&=\frac d{dt} \la \varepsilon,\ii \Lambda_1 P\ra
=-\la \ii \partial_t\varepsilon,\Lambda_1 P\ra + \la \varepsilon,\ii\partial_t \Lambda_1 P\ra\\
& = -\la -\partial_x^2\varepsilon +c^2\varepsilon+h(U+\varepsilon,V+\eta)-h(U,V),\Lambda_1 P\ra\\
&\quad +\la F,\Lambda_1 P\ra-\la \vec{m}_1\cdot\vec\md_1,\Lambda_1 P\ra-\la\vec{m}_\varphi\cdot\vec\md_\varphi, \Lambda_1 P\ra
-\la\ii\varepsilon, \partial_t (\Lambda_1 P)-\ii c^2\Lambda_1 P\ra.
\end{align*}
We claim
\begin{equation}\label{diese}
\left| \la -\partial_x^2\varepsilon+c^2 \varepsilon+h(U+\varepsilon,V+\eta)-h(U,V),\Lambda_1 P\ra\right|
\lesssim t^{-2}\log t.
\end{equation}
Observe that
\begin{align*}
h(U+\varepsilon,V+\eta)-h(U,V)
=2|U|^2 \varepsilon +U^2 \bar \varepsilon +\omega |V|^2\varepsilon
+2\omega U\Re(V\bar \eta)+O(|\varepsilon|^2+|\eta|^2).
\end{align*}
By Lemma~\ref{tech},
$\|V^2\Lambda_1 P\|_{L^2}\lesssim t^{-1} \log t$, $\|UV\Lambda_1 P\|_{L^2}\lesssim t^{-1}$,
and thus
\begin{align*}
&\left| \la -\partial_x^2\varepsilon+c^2 \varepsilon+h(U+\varepsilon,V+\eta)-h(U,V),\Lambda_1 P\ra
-\la \varepsilon, -\partial_x^2( \Lambda_1 P)+c^2\Lambda_1 P+3|P|^2 \Lambda_1 P\ra\right|\\
&\quad \lesssim t^{-1}(\log t)\left(\|\varepsilon\|_{L^2}+\|\eta\|_{L^2}\right)+ \|\varepsilon\|_{L^2}^2+\|\eta\|_{L^2}^2
\lesssim t^{-1-\theta_1}\log t .
\end{align*}
Using $\Lp (\Lambda Q)=-2Q$ from \eqref{e:s}
and
$\|\partial_x^2 (\Lambda_1 P) - \partial_1^2 (\Lambda_1P)\|_{L^2}
\lesssim |\beta_1|\lesssim t^{-1}$
(by analogy with the notation introduced in~\S\ref{S:2:1}, we set $\partial_1^2 (\Lambda_1P)=(\Lambda Q_c)''(x-\sigma_1) e^{\ii \Gamma_1}$)
 we see that
\[
\| [-\partial_x^2 (\Lambda_1 P)+c^2 \Lambda_1 P+3 |P|^2\Lambda_1 P ] + 2c^2 P\|_{L^2}\lesssim t^{-1}.
\]
Thus, by \eqref{BS} and \eqref{ePetaR}, we obtain \eqref{diese}.

The estimate $|\la F,\Lambda_1 P\ra|\lesssim e^{-2c\sigma}\lesssim t^{-2}$ is clear from \eqref{F} and then \eqref{BS}.
Next, using $\la P,\Lambda_1 P\ra=\la Q_c,\Lambda Q_c\ra=\frac 12\|Q_c\|_{L^2}^2=2c$ and $\la \ii Q_c',\Lambda Q_c\ra=\la xQ_c,\Lambda Q_c\ra=0$,
we obtain
\[
-\la \vec{m}_1\cdot\vec\md_1,\Lambda_1P\ra
=-2c(\dot \gamma_1+\dot \beta_1\sigma_1+\beta_1^2).
\]
Moreover, using Lemma~\ref{tech},
\begin{align*}
-\la \vec{m}_\varphi\cdot\vec\md_\varphi,\Lambda_1 P\ra
&=-(\dot \gamma_1+\dot \beta_1\sigma_2+\beta_1^2)\la \varphi,\Lambda_1 P\ra
+\dot \beta_1\la x_2\varphi,\Lambda_1P\ra\\
&=(|\dot \gamma_1|+|\dot \beta_1||\sigma_2|+\beta_1^2)O(\sigma^2 e^{-2c\sigma}).
\end{align*}
Last, using the analogue of \eqref{dtP} for $\Lambda_1 P$, we have
\[
|\ii \la\varepsilon, \partial_t( \Lambda_1P )-\ii c^2\Lambda_1P\ra|
\lesssim (|\dot \gamma_1|+|\dot\beta_1||\sigma_1|+|\dot\sigma_1-2\beta_1|+|\beta_1|)\|\varepsilon\|_{L^2}.
\]
The conclusion of these estimates is
\begin{equation*}
|\dot \gamma_1|\lesssim t^{-2} \log t+t^{-1} |\dot\sigma_1-2\beta_1|+|\dot\beta_1|\log t.
\end{equation*}

Proceeding similarly with the orthogonality condition $\la \eta,\ii \Lambda_2 R\ra=0$, we check
\[
|\dot \gamma_2|\lesssim t^{-\theta}+t^{-1}|\dot\sigma_2-2\beta_2|+|\dot\beta_2|\log t.
\]
Note that we again use $\Lp (\Lambda Q)=-2Q$ and \eqref{ePetaR} for $\eta$. The term $t^{-\theta}$ comes from estimate
of $G$ in \eqref{G}, which is to be compared with \eqref{F} for $F$.

Next, differentiating the orthogonality conditions $\la \varepsilon,x_1 P\ra=\la \eta,x_2R\ra=0$,
using the relation $\Lm(xQ)=-2Q'$ from \eqref{e:s} and last $\la \varepsilon,\ii \partial_1 P\ra=\la \eta,\ii \partial_2 R\ra=0$,
we find
\begin{align*}
|\dot \sigma_1-2\beta_1|&\lesssim t^{-1-\ths} \log t+t^{-1}(|\dot \gamma_1|+|\dot\beta_1||\sigma_1|+|\dot \sigma_2-2\beta_2|), \\
|\dot \sigma_2-2\beta_2|&\lesssim t^{-1-\ths} \log t+t^{-1}(|\dot \gamma_2|+|\dot\beta_2||\sigma_2|+|\dot \sigma_1-2\beta_1|) .
\end{align*}
Note that for these estimates, we have also used
$\la F,\ii x_1 P\ra =0$ and $\la G,\ii x_2 R\ra =0$.

Last, differentiating the orthogonality conditions $\la \varepsilon,\ii \partial_1 P\ra=\la \eta,\ii \partial_2 R\ra=0$,
using the relation $\Lp (Q')=0$ from \eqref{e:s} and $\la F^\perp, \partial_1P \ra=\la G^\perp,\partial_2 R\ra=0$, we check that
\begin{equation*}
|\dot \beta_1+a|+|\dot \beta_2+b|\lesssim t^{-1-\ths}+t^{-1}(|\dot \gamma_1|+|\dot \sigma_1-2\beta_1|+|\dot \gamma_2|+|\dot \sigma_2-2\beta_2|).
\end{equation*}

The proof of \eqref{param}-\eqref{m}-\eqref{beta} follows from  the above estimates and \eqref{est:ab}.
\end{proof}
\subsection{Energy estimates}\label{S:ener}
Let
 \[
 H(u,v)=\dfrac{1}{4}|u|^4 +\dfrac{1}{4}|v|^4+\dfrac{\omega}{2}|u|^2 |v|^2 , \quad h(u,v)=(|u|^2+\omega|v|^2)u.
 \]
 and remark that
 \begin{align*}
 & d_1H(U,V)(\varepsilon)=\dfrac{1}{2}(|U|^2+\omega|V|^2)(U\bar{\varepsilon}+\bar{U}\varepsilon)=\Re\left(h(U,V)\varepsilon \right),\\
 & d_2H(U,V)(\eta)=\dfrac{1}{2}(|V|^2+\omega|U|^2)(V\bar{\eta}+\bar{V}\eta)=\Re\left(h(V,U)\eta \right),\\
 & d_1h(U,V)(\varepsilon) = 2|U|^2\varepsilon +U^2\bar{\varepsilon} +\omega |V|^2 \varepsilon, \quad d_2h(U,V)(\eta)=\omega (V\bar{\eta}+\bar{V}\eta)U,\\
 & \frac 12 (\varepsilon, \eta)^\textup{T}(d^2h)(U,V)(\varepsilon, \eta)=2 \varepsilon \Re (U\bar{\varepsilon})+U|\varepsilon|^2
 +2 \omega \varepsilon \Re( V \bar{\eta})+\omega U|\eta|^2.
 \end{align*} 
 Consider the energy functional for $\big(\begin{smallmatrix}\varepsilon\\ \eta\end{smallmatrix}\big)$
 \begin{multline*}
 \mathbf{K}(t,\varepsilon,\eta)=\dfrac{1}{2}\int \big\{|\partial_x \varepsilon|^2+|\partial_x \eta|^2-2 \big[ H(U+\varepsilon, V+\eta) - H(U, V) \\
 - d_1H(U, V)(\varepsilon) -d_2H(U, V)(\eta) \big] \big\}
 \end{multline*}
 and the mass functionals for $\varepsilon$ and $\eta$
 \[
 \mathbf{M}=M_1+M_2,\quad M_1(\varepsilon)=\dfrac{c^2}{2}\int |\varepsilon|^2, \quad M_2(\eta)=\dfrac{1}{2}\int |\eta|^2.
 \]
Let $ \chi:[0,+\infty) \to [0,+\infty)$ be a smooth non-increasing function satisfying $\chi \equiv 1$ on $[0,\frac{1}{4}]$ and $\chi \equiv 0$ on $[\frac{1}{2},+\infty)$. Denote $\mathbf{J}=J_1+J_2$ where, for $j=1,2$,
\[
J_j(t,\varepsilon,\eta)=\beta_j \, \Im \int \left[(\partial_x\varepsilon) \bar{\varepsilon}+(\partial_x\eta)\bar{\eta}\right]\chi_j
\quad\mbox{where}\quad \chi_j(t,x)=\chi\left(\frac{|x-\sigma_j(t)|}{\log t}\right).\] 
Last, we set
\[
\mathbf S(t,\varepsilon,\eta)=\la \varepsilon,F^\perp\ra + 2\beta \la \varepsilon,\ii \phi\ra
+ \la \eta,G^\perp\ra\quad \mbox{where}\quad
\phi=\partial_1 \varphi-c\varphi.
\]
Last, set
\[
\mathbf{W}(t,\varepsilon, \eta)=\mathbf{K} (t,\varepsilon, \eta) + \mathbf{M}(t,\varepsilon, \eta)- \mathbf{J}(t,\varepsilon, \eta)
-\mathbf{S}(t,\varepsilon, \eta).
\]
We refer to \cite{KMR,Martel1,MMnls,MMT2,MRlog,NV1,RaSz11} for similar energy functionals.
However, the introduction of the correcting term $\mathbf{S}$ seems to be a previously unnoticed general improvement of the energy method
in this context. See Section~\ref{discuss}.

Under the bootstrap~\eqref{BS}, we prove the following estimates.
\begin{proposition}\label{prop:coer}
Let $\ths<\theta<\min\{\frac 1c;2\}$.
It holds
\begin{equation}
\label{'eq3.25} 
\|\varepsilon \|_{H^1}^2 + \| \eta \|_{H^1}^2\lesssim 
\mathbf{W}(t, \varepsilon , \eta ) + C t^{-2\theta},
\end{equation}
and
\begin{align}
\label{eq3.26}
\left| \dfrac{d}{dt} [\mathbf{W}(t,\varepsilon , \eta )] \right| \lesssim t^{-1-2\ths}(\log t)^{-1}.
\end{align}
\end{proposition}
\begin{proof}[Proof of Proposition \ref{prop:coer}]
\textbf{step 1.} The coercivity property \eqref{'eq3.25} is a consequence of the coercivity property around one solitary wave in Lemma \ref{lem:0}, the orthogonality relations \eqref{ortho}-\eqref{ePetaR}) and the positivity of $\Lc$.
It also involves a localization argument similar to the proof of Lemma 4.1 in \cite{MMT2} for the scalar case.
 
Note that by \eqref{BS},
\[|\mathbf{J}(t,\varepsilon, \eta)| \lesssim t^{-1}(\|\varepsilon\|_{H^1}^2+\|\eta \|_{H^1}^2)\]
 and by \eqref{F} and \eqref{G},
\[
|\mathbf{S}(t,\varepsilon, \eta)|\lesssim t^{-\theta}(\|\varepsilon\|_{H^1}+\|\eta\|_{H^1}).
\]

Next, we see that the following terms in the functional $\mathbf{K}$ are easily controlled
\[
\int (|P\varphi|+|\varphi|^2)|\varepsilon|^2+\int |U \varepsilon| |V\eta|\lesssim t^{-1}\left( \|\varepsilon\|_{H^1}^2+ \|\eta\|_{H^1}^2\right).
\]
Moreover, cubic and higher order terms in $\varepsilon$ or $\eta$ are of order $t^{-\ths}\left( \|\varepsilon\|_{H^1}^2+ \|\eta\|_{H^1}^2\right)$. 

Therefore, we are reduced to consider the following two decoupled functionals
\begin{align*}
\mathbf{W}_1&= \int \left\{|\partial_x \varepsilon|^2+c^2 |\varepsilon|^2 -|P|^2|\varepsilon|^2-2[\Re(P\bar \varepsilon)]^2
-\omega |R|^2 |\varepsilon|^2\right\},\\
\mathbf{W}_2&= \int \left\{|\partial_x \eta|^2+ |\eta|^2 -|R|^2|\eta|^2-2[\Re(R\bar \eta)]^2
-\omega |P|^2 |\eta|^2\right\}.
\end{align*}
We focus on the coercivity property for $\mathbf{W}_1$, the case of $\mathbf{W}_2$ is similar.

Denote $\Phi : \RR \to \RR$ an even function of class $\mathcal{C}^2$ such that
\[
\Phi \equiv 1 \text{ on } [0,1], \quad \Phi \equiv e^{-x} \text{ on }[2,+\infty),
\quad 
e^{-x} \le \Phi(x) \le e^{-3x}, \quad \Phi' \le 0 \text{ on }\RR.
\]
Let $B > 1$ and $\Phi_B(x) = \Phi(x/B)$. We claim that for $B$ large enough, there exists $\mu_1>0$, such that for any $\hat\varepsilon$ satisfying
$\la \hat\varepsilon,Q\ra=\la \hat\varepsilon,xQ\ra=\la \hat\varepsilon,\ii\Lambda Q\ra=0$, and any $\tilde\varepsilon$, it holds
\begin{align*}
\mathcal{N}_1 (\hat{\varepsilon} )
&:=\int \Phi_B \left\{ |\partial_x \hat{\varepsilon}|^2+|\hat{\varepsilon}|^2 - Q^2|\hat{\varepsilon}|^2-2\left[\Re( Q\bar{\hat{\varepsilon}})\right]^2 \right\} \geq \mu_1 \int \Phi_B (|\partial_x \hat{\varepsilon}|^2+|\hat{\varepsilon}|^2),\\
\mathcal{N}_{2} (\tilde{\varepsilon} ) &
:= \int \Phi_B\left\{|\partial_x \tilde{\varepsilon}|^2+c^2|\tilde{\varepsilon}|^2 - {\omega} Q^2|\tilde{\varepsilon}|^2 \right\} 
\geq \mu_1 \int \Phi_B (|\partial_x \tilde{\varepsilon}|^2+|\tilde{\varepsilon}|^2).
\end{align*}
Setting $z=\hat{\varepsilon}\Phi_B^{\frac 12}$ and following the proof of Claim~8 in \cite{MMT2}, the coercivity of $\mathcal{N}_1$
follows from (i) of Lemma~\ref{lem:0} applied to the function $z$.
A similar localization argument, using the coercivity property of $\Lc$ proves the estimate for $\mathcal{N}_2(\tilde{\varepsilon} )$
without any orthogonality condition on $\tilde{\varepsilon}$.
This is where our proof needs the condition \eqref{omega}.

Using these estimates with $\hat\varepsilon$ and $\tilde\varepsilon$ such that $ \varepsilon=c \hat \varepsilon( c(x-\sigma_1)) e^{\ii \Gamma_1}$ and
$ \varepsilon=\tilde\varepsilon(x-\sigma_2) e^{\ii \Gamma_2}$, the orthogonality conditions \eqref{ortho}
and the almost orthogonality relation \eqref{ePetaR}, we obtain the estimate
$\|\varepsilon\|_{H^1}^2\lesssim \mathbf{W}_1 + t^{-4} (\log t)^2$.

\smallskip 

\noindent\textbf{step 2.} Time variation of the energy. Denote 
\begin{align*}
K_1=h(U+\varepsilon, V+\eta) - h(U, V) - d_1h(U,V)(\varepsilon) - d_2h(U,V)(\eta),
\end{align*}
\[
K_2=h(V+\eta, U+\varepsilon) - h(V, U) - d_1h(V,U)(\eta) - d_2h(V,U)(\varepsilon),
\]
so that
\begin{align*}
K_1&=\dfrac12 {(\varepsilon, \eta)^\textup{T} (d^2h)(U, V)(\varepsilon, \eta)}+O(|\varepsilon|^3+|\eta|^3),\\
K_2&=\dfrac12 {(\eta, \varepsilon)^\textup{T} (d^2h)(V, U)(\eta, \varepsilon)}+O(|\varepsilon|^3+|\eta|^3).
\end{align*}
We prove the following estimate
\begin{equation}\label{eq3.31}
\begin{aligned}
 \dfrac{d}{dt} [\mathbf{K}(t,\varepsilon, \eta)]&= 2\beta_1 \langle \partial_x U, K_1 \rangle 
 + 2\beta_2 \langle \partial_x V, K_2 \rangle
-c^2\langle \ii U, K_1 \rangle-\langle \ii V, K_2\rangle \\
&\quad 
-\langle \ii D_\varepsilon\mathbf{K} , \mathcal{E}_{U} \rangle 
-\langle \ii D_\eta \mathbf{K} , \mathcal{E}_{V}\rangle +
O(t^{-1-2\ths}(\log t)^{-1}).\end{aligned}
\end{equation}
The time derivative of $t \mapsto \mathbf{K} (t,\varepsilon(t),\eta(t))$ splits into three parts
\[
\dfrac{d}{dt} [\mathbf{K}(t, \varepsilon, \eta]=D_t\mathbf{K} (t, \varepsilon, \eta)
+\langle D_\varepsilon \mathbf{K}(t, \varepsilon, \eta), \partial_t \varepsilon \rangle
+ \langle D_\eta \mathbf{K}(t, \varepsilon, \eta), \partial_t\eta \rangle,
\]
where $D_t$ denotes the differentiation of $\mathbf{K}$ with respect to $t$, and $D_\varepsilon, D_\eta$ the differentiation of $\mathbf{K}$ with respect to $\varepsilon$ and $\eta$. 
In particular, $D_t\mathbf{K} =- \langle \partial_t U, K_1 \rangle -\langle \partial_t V, K_2 \rangle$.

We claim
\begin{equation}\label{sur:UV}
\begin{aligned}
& \partial_t U = \ii c^2 U -2\beta_1 \partial_x U + O_{H^1}(t^{-\theta}),\\
& \partial_t V = \ii V -2\beta_2 \partial_x V + O_{H^1}(t^{-\theta}).
\end{aligned}
\end{equation}
Indeed, from the definition of $U$
\begin{align*}
\partial_t U &= \ii c^2 U -2\beta_1 \partial_1 U
-(\dot \sigma_1 -2\beta_1)\partial_1 P +\ii (\dot \gamma_1+\dot \beta_1\sigma_1) P
+\ii \dot \beta_1 x_1 P
 \\
&\quad -(\dot \sigma_2 -2\beta_1)\partial_1 \varphi+\ii (\dot\gamma_1+\dot \beta_1\sigma_2+\ii c\dot \sigma) \varphi
+\ii \dot \beta_1x_2\varphi.
\end{align*}
Thus, using \eqref{param} and \eqref{beta}, we obtain \eqref{sur:UV} for $U$.
The proof for $V$ is similar.

Using \eqref{sur:UV} and \eqref{BS}, we obtain
\begin{equation*}
D_t \HH (t, \varepsilon, \eta) = 2\beta_1 \langle \partial_x U, K_1 \rangle +2\beta_2 \langle \partial_x V, K_2 \rangle
- c^2 \langle \ii U , K_1 \rangle - \langle \ii V , K_2 \rangle
 + O (t^{-\theta-2\ths}).
\end{equation*}
Next, we observe 
\begin{equation*}
D_\varepsilon \HH (t,\varepsilon, \eta)=-\partial_x^2 \varepsilon - h(U+\varepsilon, V+\eta) +h(U,V)
\end{equation*}
so that the equation of $\varepsilon$ in \eqref{eq:ee} rewrites
$\ii \partial_t\varepsilon - D_\varepsilon \HH(t, \varepsilon, \eta) +\mathcal{E}_U= 0$ and thus
\begin{equation*}
\langle D_\varepsilon \HH (t, \varepsilon, \eta), \partial_t\varepsilon \rangle = -\langle \ii D_\varepsilon \HH (t,\varepsilon, \eta), \mathcal{E}_U \rangle.
\end{equation*}
Similarly,
\[
\langle D_\eta \HH (t, \varepsilon, \eta), \partial_t\eta \rangle = -\langle \ii D_\eta \HH (t,\varepsilon, \eta), \mathcal{E}_V \rangle.
\]
We have proved \eqref{eq3.31}. 
\smallskip

\noindent \textbf{step 3.} Time variation of the total mass. We claim
\begin{equation}
\label{eq3.32}
\begin{aligned}
\dfrac{d}{dt}[\MM( \varepsilon,\eta)]= c^2\langle \ii U, K_1 \rangle+ \langle \ii V, K_2 \rangle
-\langle \ii c^2 \varepsilon, \mathcal{E}_U \rangle - \langle \ii \eta, \mathcal{E}_V \rangle.\end{aligned}
\end{equation}
By integration by parts, we have $\langle \ii \partial_x^2 \varepsilon, \varepsilon \rangle=0$ so
from \eqref{eq:ee}, 
\[
\dfrac{d}{dt}[M_1(\varepsilon)] 
=c^2 \langle \partial_t\varepsilon, \varepsilon \rangle=- c^2 \langle \ii\varepsilon, h(U+\varepsilon, V+\eta)-h(U,V) \rangle- c^2\langle \ii\varepsilon, \mathcal{E}_{U} \rangle.
\]
We claim the following identity
\begin{align}
\label{eq3.33}
\langle \ii U, K_1 \rangle + \langle \ii \varepsilon, h(U+\varepsilon, V+\eta) - h(U,V) \rangle=0.
\end{align}
Indeed, since $h(u,v) \overline{u}$ is real, for all $\theta \in \RR$, it holds
\[
\langle \ii (U+\theta \varepsilon), h(U +\theta \varepsilon, V + \theta \eta) \rangle = 0.
\]
Differentiating with respect to $\theta$, and taking $\theta = 0$, we obtain
\[
\langle \ii \varepsilon, h(U,V) \rangle + \langle \ii U, d_1h (U , V) (\varepsilon ) \rangle 
+\langle \ii U , d_2h (U, V ) ( \eta ) \rangle = 0
\]
Moreover, with $\theta =0$ and $\theta = 1$ 
\[
\langle i U , h (U, V ) \rangle = 0,\quad
\langle i (U + \varepsilon ), h (U + \varepsilon , V + \eta ) \rangle = 0 .
\]
We see that \eqref{eq3.33} follows from combining these identities.

This yields 
$\frac{d}{dt}M_1 = c^2\langle \ii U , K_1 \rangle - c^2\langle \ii \varepsilon , \mathcal{E}_U \rangle$. Computing also $\frac{d}{dt}M_2 $, we obtain
\eqref{eq3.32}. 

\smallskip

\noindent \textbf{step 4.} Time variation of the localized momentum. We claim
 \begin{equation} \label{eq3.34}
\dfrac{d}{dt} [\mathbf{J}(t, \varepsilon , \eta ) ]
= 2 \beta_1 \langle \partial_x U , K_1 \rangle +2 \beta_2 \langle \partial_x V, K_2 \rangle +O(t^{-1-2\ths}(\log t)^{-1}).
 \end{equation}
By direct computation,
 \begin{align*}
 \dfrac{d}{dt} [ J_1 (t, \varepsilon , \eta)] & 
 = \dot{\beta }_1 \Im \int [ (\partial_x\varepsilon) \overline{\varepsilon} +(\partial_x\eta) \overline{\eta}] \chi_1
 +\beta_1 \Im \int [ (\partial_x\varepsilon) \overline{\varepsilon} +(\partial_x\eta) \overline{\eta}] \partial_t\chi_1 
 \\
 &\quad +\beta_1 \langle \ii \partial_t\varepsilon , 2 \chi_1 \partial_x \varepsilon + \varepsilon \partial_x \chi_1 \rangle
+\beta_1 \langle \ii \partial_t\eta , 2 \chi_1 \partial_x \eta + \eta \partial_x \chi_1 \rangle.
 \end{align*}
 By \eqref{BS} and \eqref{beta}, we have 
 \[
\left | \dot{ \beta }_1 \int [ (\partial_x\varepsilon) \overline{\varepsilon} +(\partial_x\eta) \overline{\eta}] \chi_1 \right | \lesssim t ^ {-2} \left(\| \varepsilon \| ^ 2 _ {H ^ 1 }+ \|\eta\|_{H^1}^2\right)
\lesssim t^{-2-2\ths}.
\]
By direct computations,
\[
\partial_t \chi_j(t,x)=
-\left[\frac{\dot\sigma_j}{\log t}\frac{x-\sigma_j}{|x-\sigma_j|}
+\frac{|x-\sigma_j|}{t(\log t)^2}\right] \chi'\left(\frac{|x-\sigma_j|}{\log t}\right)
\]
and so by \eqref{BS}, \eqref{param} and the properties of $\chi$,
$
|\partial_t \chi_j| \lesssim t^{-1} (\log t) ^ {-1}$.
It follows that
\[ \left | \beta_1 \Im \int [ (\partial_x\varepsilon) \overline{\varepsilon} +(\partial_x\eta) \overline{\eta}] \partial_t\chi_1 \right | \lesssim t ^ {-2} (\log t) ^ {-1} \left( \| \varepsilon\| ^2 _ {H ^1 } + \|\eta\|_{H^2}^2\right)
\lesssim t^{-2-2\ths}(\log t) ^ {-1} .
\]

Next, using the equation \eqref{eq:ee}
\begin{align*}
\langle \ii \partial_t \varepsilon , 2 \chi_1 \partial_x \varepsilon + \varepsilon \partial_x \chi_1 \rangle
&= -\langle \partial_x^2 \varepsilon , 2 \chi_1 \partial_x \varepsilon + \varepsilon \partial_x \chi_1 \rangle\\
&\quad-\la h(U+\varepsilon,V+\eta)-h(U,V),2 \chi_1 \partial_x \varepsilon + \varepsilon \partial_x \chi_1\ra\\
&\quad- \la \EE_U , 2 \chi_1 \partial_x \varepsilon + \varepsilon \partial_x \chi_1\ra.
\end{align*}
Integrating by parts, we have
\[ -\langle \partial_x^2 \varepsilon , 2 \chi_1 \partial_x \varepsilon + \varepsilon \partial_x \chi_1 \rangle = \int |\partial_x \varepsilon|^2 \partial_x \chi_1 - \dfrac{1}{2} \int |\varepsilon | ^2 \partial_x^3 \chi_1 .
\]
Since $|\partial_x \chi_1| \lesssim (\log t)^{-1}$ and $|\partial_x^3 \chi_1| \lesssim (\log t)^{-3}$,
from \eqref{BS}, we have
\[ \left | \beta_1 \langle \partial_x^2 \varepsilon , 2 \chi_1 \partial_x \varepsilon + \varepsilon \partial_x \chi_1 \rangle \right | \lesssim t ^ {-1-2\ths} (\log t)^{-1}.
\]
For the term containing $\mathcal{E}_{U}$, we use \eqref{on:UV}, \eqref{BS}, \eqref{F} and \eqref{m},
\[
\bigg| \beta_1\langle \mathcal{E}_{U}, 2\chi_1 \partial_x \varepsilon +\varepsilon \partial_x \chi_1 \rangle \bigg| \lesssim t^{-1-\theta}\|\varepsilon \|_{H^1} \lesssim t^{-1-2\ths}(\log t)^{-1}.
\]

Then, we estimate, using $|\partial_x \chi_1| \lesssim (\log t)^{-1}$,
\begin{equation*}
|\langle h(U+\varepsilon, V+\eta)-h(U,V), \varepsilon \partial_x \chi_1 \rangle|
 \lesssim (\log t)^{-1} \left(\|\varepsilon\|_{H^1}^2+\|\eta\|_{H^1}^2\right)
 \lesssim t^{-2\ths}(\log t)^{-1}.
\end{equation*}
Collecting the above estimates, we obtain
\begin{align*}
\dfrac{d}{dt} [J_1(t,\varepsilon, \eta)] 
&= -2\beta_1 \langle  \chi_1 \partial_x \varepsilon, h(U+\varepsilon, V+\eta)-h(U,V) \rangle\\
&\quad 
-2\beta_1 \langle  \chi_1 \partial_x \eta , h(V+\eta,U+\varepsilon)-h(V,U) \rangle
+ O(t^{-1-2\ths}(\log t)^{-1}).
\end{align*}

We complete the proof of \eqref{eq3.34} by showing the following
\begin{equation}\label{idee}\begin{aligned}
\langle \partial_x U, K_1 \rangle &+\langle \chi_1 \partial_x  \varepsilon, h(U+\varepsilon, V+\eta)-h(U,V)\rangle\\
&  +\langle  \chi_1 \partial_x  \eta, h( V+\eta,U+\varepsilon)-h(V,U) \rangle
=O(t^{-2\ths}(\log t)^{-1}).
\end{aligned}\end{equation}
First, we prove the identity
\begin{equation}
\label{eq3.35}
\begin{aligned}
 &\langle \partial_x U,K_1 \rangle + \langle \partial_x \varepsilon, h(U+\varepsilon, V+\eta) - h(U,V) \rangle \\&
 + \langle \partial_x V, K_2\rangle+ \langle \partial_x \eta, h(V+\eta, U+\varepsilon) - h(V,U) \rangle=0.
\end{aligned}
\end{equation} 
Indeed,   we have
\[
\langle \partial_x u, h(u,v) \rangle + \langle \partial_x v, h(v,u) \rangle=\int \partial_x [H(u,v)]=0.
\]
Applying this to $u=U+\theta \varepsilon$ and $v=V+\theta \eta$, we have that for all $\theta \in \RR$
\[
\langle \partial_x (U+\theta \varepsilon), h(U+\theta \varepsilon, V+\theta \eta) \rangle + \langle \partial_x (V+\theta \eta), h(V+\theta \eta, U+\theta\varepsilon) \rangle=0.
\]
Taking the derivative with respect to $\theta$ at $\theta = 0$, we obtain
\begin{align*}
\langle \partial_x \varepsilon , h(U,V)\rangle 
+ \langle \partial_x \eta, h(V,U) \rangle &
+ \langle \partial_x U, d_1h(U,V)(\varepsilon) \rangle + \langle \partial_x U, d_2h(U,V)(\eta) \rangle\\
&+ \langle \partial_x V, d_1h(V,U)(\eta) \rangle + \langle \partial_x V, d_2h(V,U)(\varepsilon) \rangle=0.
\end{align*}
Moreover, using the above identity with $ \theta = 0$ and $\theta = 1$, we have
\begin{align*}
& \langle \partial_x U, h(U,V) \rangle+\langle \partial_x V, h(V,U) \rangle=0,\\
& \langle \partial_x (U+\varepsilon), h(U+\varepsilon, V+\eta) \rangle + \langle \partial_x (V+\eta), h(V+\eta, U+\varepsilon) \rangle=0.
\end{align*}
 Gathering these identities, we obtain \eqref{eq3.35}.
 
We apply identity \eqref{eq3.35} to $\chi_1^{\frac 14}U$, $\chi_1^{\frac 14}V$, $ \chi_1^{\frac 14}\varepsilon$ and $\chi_1^{\frac 14}\eta$.
Recall that $|\partial_x \chi_1|\lesssim (\log t)^{-1}$ and also note that by the definition of $\chi$,
$|\chi_1 V|+(1-\chi_1) |\partial_x U| \lesssim (\log t)^{-1}$.
In particular, this shows that
\[
|\langle \partial_x (\chi_1^{\frac 14} U),K_1 \chi_1^{\frac 34}\rangle-\langle \partial_x U,K_1 \rangle|
+|\langle \partial_x (\chi_1^{\frac 14}V), K_2\chi_1^\frac34\rangle |
=O(t^{-2\ths}(\log t)^{-1}),
\]
\[
\langle [\chi_1^\frac 34\partial_x (\chi_1^\frac 14 \varepsilon)-\chi_1 \partial_x \varepsilon ], h(U+\varepsilon, V+\eta) - h(U,V) \rangle
=O(t^{-2\ths}(\log t)^{-1}),
\]
and
\[
\langle [\chi_1^\frac 34\partial_x (\chi_1^\frac 14 \eta)-\chi_1 \partial_x \eta], h(V+\eta,U+\varepsilon) - h(V,U) \rangle
=O(t^{-2\ths}(\log t)^{-1}).
\]
This proves \eqref{idee} and then \eqref{eq3.34}, the computations for $J_2$ being identical.

\smallskip

\noindent\textbf{step 5.} Additional correction terms.
We claim
\begin{equation}\label{on:S}
\frac d{dt} [\mathbf{S}(t, \varepsilon , \eta ) ] = 
-\la \ii (D_\varepsilon\mathbf{K}+c^2\varepsilon), F^\perp+2\ii \beta\phi\ra
-\la \ii (D_\eta\mathbf{K}+\eta), G^\perp\ra+O(t^{-(1+\theta+\theta_1)}).
\end{equation}
We compute, using \eqref{eq:ee},
\begin{align*}
\frac d{dt}\la \varepsilon,F^\perp\ra
= - \la \ii (D_\varepsilon\mathbf{K}+c^2\varepsilon), F^\perp\ra-\la \EE_U,\ii F^\perp\ra+\la \varepsilon,\partial_t F^\perp-\ii c^2 F^\perp\ra.
\end{align*}
From \eqref{onFFGG} and $F^\perp e^{-\ii\Gamma_1}\in \RR$, it follows that $\la \vec{m}_1^a\cdot \vec\md_1,\ii F^\perp\ra=0$.
One also observes that
\[
\la \vec{m}_\varphi \cdot \md_\varphi,\ii F^\perp\ra=
c\dot\sigma\la\varphi,F^\perp\ra
+(\dot\sigma_2-2\beta_1)\la \partial_1 \varphi,F^\perp\ra
=O(t^{-5}(\log t)^2)=O(t^{-(1+2\theta)},
\]
where we have used \eqref{param} and (from Lemma~\ref{tech} and the definitions of $F^\perp$ and $\varphi$)
\begin{equation}\label{phiF}
|\la \varphi,F^\perp\ra|+|\la \partial_1 \varphi,F^\perp\ra|\lesssim t^{-4}(\log t)^{2}.
\end{equation}
Since $\la F^\perp,\ii F^\perp\ra=0$, it follows that $\la \EE_U,\ii F^\perp\ra=O(t^{-(1+2\theta)})$.
Last, it follows from \eqref{dtF}, \eqref{param} and \eqref{BS} that
\[
|\la \varepsilon,\partial_t F^\perp-\ii c^2 F^\perp\ra|
\lesssim 
t^{-3-\theta_1}\lesssim t^{-1-\theta_1-\theta}.
\]
Thus, using \eqref{new:EF},
\[\frac d{dt}\la \varepsilon, F^\perp\ra=- \la\ii (D_\varepsilon \mathbf{K}+c^2 \varepsilon), F^\perp\ra+O(t^{-1-\theta_1-\theta}).\]
From \eqref{dtG} and similar estimates, we also obtain 
\[\frac d{dt}\la \eta,G^\perp\ra=- \la\ii (D_\eta\mathbf{K}+\eta), G^\perp\ra+O(t^{-1-\theta_1-\theta})\]

Finally, we compute
\begin{equation*}
\frac d{dt}\left[2\beta \la \varepsilon, \ii \phi\ra\right] 
 = 2\dot \beta \la \varepsilon, \ii \phi\ra
+2\beta \la \partial_t\varepsilon, \ii \phi\ra+
2\beta\la \varepsilon,\ii \partial_t \phi\ra.
\end{equation*}
The first term is estimated  $|\dot \beta \la \varepsilon, \ii \phi\ra|\lesssim t^{-3} \|\varepsilon\|_{H^1} \lesssim t^{-3-\ths}$ using \eqref{beta}.
Then, using \eqref{eq:ee},
\begin{equation*}
\la \partial_t\varepsilon, \ii \phi\ra+
 \la \varepsilon,\ii \partial_t \phi\ra
 =-\la \ii( D_\varepsilon\mathbf{K}+c^2\varepsilon),\ii\phi\ra
 +\la  \EE_U, \phi\ra
- \la \ii \varepsilon,\partial_t\phi-\ii c^2\phi\ra.
\end{equation*}
From \eqref{phiF},  $|\beta\la  F^\perp, \phi\ra|\lesssim t^{-5}(\log t)^2\lesssim t^{-1-2\theta}$.
From \eqref{m}, the expression of $\vec{m}_1^a\cdot\vec\md_1$ and Lemma~\ref{tech},
\[
|\beta\la  \vec{m}_1^a\cdot\vec\md_1, \phi\ra|\lesssim 
t^{-1}|\vec{m}_1^a| \left(|\la P,\phi\ra|
+|\la x_1P,\phi\ra|\right)
\lesssim t^{-3-\theta}(\log t)^2\lesssim t^{-1-\theta_1-\theta}.
\]
Next, from~\eqref{param}, the expression of $\vec{m}_\varphi\cdot\vec\md_\varphi$ and Lemma~\ref{tech},
\[
|\beta\la  \vec{m}_\varphi\cdot\vec\md_\varphi, \phi\ra|\lesssim 
t^{-1} \left(|\dot \gamma_1+\dot \beta_1\sigma_2+\beta_1^2||\la \varphi,\phi\ra|
+ |\dot \beta_1+a||\la x_2\varphi,\phi\ra|\right) \lesssim t^{-3-\theta}\lesssim t^{-1-\theta_1-\theta}.
\]

Last, using \eqref{dtP} and \eqref{param}, 
\[
|\beta\la \ii \varepsilon,\partial_t\phi-\ii c^2\phi\ra|\lesssim
t^{-3} \|\varepsilon\|_{L^2} \lesssim t^{-3-\ths}.
\]
Estimate~\eqref{on:S} is now proved.

\smallskip
 
\noindent\textbf{step 6.} Conclusion.
Combining the estimates \eqref{eq3.31}, \eqref{eq3.32}, \eqref{eq3.34}, \eqref{on:S}
and using the decompositions of $\EE_U$ and $\EE_V$ in \eqref{new:EF}, we have obtained
\begin{align*}
\frac d{dt} \mathbf{W}(t,\varepsilon,\eta)
&=
\la \ii (D_\varepsilon \mathbf{K}+c^2\varepsilon), \vec{m}_1^a\cdot \vec\md_1\ra
+ \la \ii (D_\varepsilon \mathbf{K}+c^2\varepsilon),2\ii \beta\phi+\vec{m}_\varphi\cdot \vec\md_\varphi\ra
\\&\quad 
+\la \ii (D_\eta \mathbf{K}+\eta), \vec{m}_2^b\cdot \vec\md_2\ra+O(t^{-1-2\ths}(\log t)^{-1}).
\end{align*}
We claim
\begin{equation}\label{diese2}
|\la \ii (D_\varepsilon \mathbf{K}+c^2\varepsilon), \vec{m}_1^a\cdot \vec\md_1\ra|
\lesssim t^{-(1+\ths+\theta)}.
\end{equation}
Indeed, following the proof of \eqref{diese}, using Lemma~\ref{tech}, the relations \eqref{e:s}, \eqref{BS} and the third orthogonality condition in \eqref{ortho},
it holds
\begin{align*}
&\left|\la -\partial_x^2\varepsilon+c^2\varepsilon-h(U+\varepsilon,V+\eta)+h(U,V), \partial_1P \ra\right|
\lesssim t^{-(1+\ths)},\\
&\left|\la -\partial_x^2\varepsilon+c^2\varepsilon-h(U+\varepsilon,V+\eta)+h(U,V),\ii P \ra\right|
\lesssim t^{-(1+\ths)},\\
&\left|\la -\partial_x^2\varepsilon+c^2\varepsilon-h(U+\varepsilon,V+\eta)+h(U,V),\ii x_1P \ra\right|
\lesssim t^{-(1+\ths)}\log t.
\end{align*}
Thus, \eqref{diese2} follows from \eqref{m} and \eqref{beta}.
Similarly,
\[
|\la \ii (D_\eta \mathbf{K}+\eta), \vec{m}_2^b\cdot \vec\md_2\ra|
\lesssim t^{-(1+\ths+\theta)}.
\]

Finally, we remark that  from the explicit expression of
$\vec{m}_\varphi\cdot \vec\md_\varphi$ and ~\eqref{param}
\[
\|2\ii \beta\phi+\vec{m}_\varphi\cdot \vec\md_\varphi\|_{H^1}\lesssim t^{-1-\theta},
\]
which implies by integration by parts and then \eqref{BS}
\[
|\la \ii (D_\varepsilon \mathbf{K}+c^2\varepsilon),2\ii \beta\phi+\vec{m}_\varphi\cdot \vec\md_\varphi\ra|
\lesssim t^{-1-\ths-\theta}.
\]
The proof of Proposition~\ref{prop:coer} is complete.
\end{proof}

\subsection{Bootstrap argument}
\begin{proposition}\label{pr:boot}
There exists $\tzero>1$ large enough and for any $\tin\geq \tzero$, there exists $\sigma_\infty$ satisfying \eqref{eq:siginf}
such that the solution $\big(\begin{smallmatrix} u\\v\end{smallmatrix}\big)$ of \eqref{snls} corresponding to initial data
$\big(\begin{smallmatrix} U \\ V\end{smallmatrix}\big)(\tin)$ at $t=\tin$ with parameters chosen as in \eqref{initial}-\eqref{eq:betainf}
admits a decomposition \eqref{decomposition}-\eqref{ortho} which satisfies \eqref{BS} on $[\tzero,\tin]$.
Moreover, $| \gamma_1|+| \gamma_2|\lesssim t^{1-\theta_1}$ on $[\tzero,\tin]$.
\end{proposition}
\begin{proof}
For $\tzero$ large enough, for any $\tin\geq \tzero$ and any $\sigma_\infty$ satisfying \eqref{eq:siginf}, we define
\[
\tstar=\tstar(\tin,\sigma_\infty)=\inf\{t\in [\tzero,\tin] \mbox{ such that \eqref{BS} holds on $[t,\tin]$}\} \in [\tzero,\tin].
\]
We prove by contradiction that, provided $\tzero$ is large enough independent of $\tin$, there exists at least a value of $\sigma_\infty$ satisfying \eqref{eq:siginf} such that $\tstar=\tzero$. We work only on the time interval $[\tstar,\tin]$ on which the boostrap estimates \eqref{BS} hold.

First, we strictly improve the estimates of $\varepsilon$ and $\eta$ in \eqref{BS}.
Indeed, integrating \eqref{eq3.26} on $[t,\tin]$ and using \eqref{'eq3.25}, it holds
\[
\|\varepsilon\|_{H^1}^2+\|\eta\|_{H^1}^2\lesssim t^{-2\ths}(\log t)^{-1},
\]
which strictly improves the estimate in \eqref{BS} for large $t$.

Next, we close the estimates on $\beta_1$, $\beta_2$ and $\beta$ in \eqref{BS}.
Using the estimate of $\sigma$ in \eqref{BS}, \eqref{beta}, \eqref{est:ab} and the expression of $\Omega_c$, it holds
\[
\left| \dot \beta_1 +\frac {1}{2c(c+1) t^2}\right|\lesssim t^{-1-\thc}.
\]
At $\tin$, we remark that by \eqref{eq:betainf} and \eqref{eq:siginf},
\[
\left|\beta_\infty-\frac1{2c\tin}\right|\lesssim \tin^{-\thc} \quad\mbox{and so}\quad
\left|\beta_1(\tin)-\frac1{2c(c+1)\tin}\right|\lesssim \tin^{-\thc}.
\]
Integrating on $[t,\tin]$ and using \eqref{initial} for $\beta_1$, we obtain
\[
\left| \beta_1 -\frac {1}{2c(c+1) t}\right|\lesssim t^{-\thc},
\]
which strictly improves \eqref{BS} for $\beta_1$ provided that $t$ is large enough.
Improving the estimate for $\beta_2$ (and then $\beta$) is similar.

Then, using \eqref{param}, we find
\[
\left| \dot \sigma_1 - \frac {1}{c(c+1) t}\right|\lesssim t^{-\thc}.
\]
Integrating on $[t,\tin]$, using \eqref{initial} and \eqref{eq:siginf} 
we obtain
\[
\left|\sigma_1 - \frac{\log (\Omega_c t)}{c(c+1)} \right|
\lesssim t^{1-\theta_2},
\]
which strictly improves the  estimate in \eqref{BS}.
The estimate on $\sigma_2$ is improved similarly.

We only have to improve the estimate on $\sigma$ to finish the bootstrap argument. This is where we need to argue by contradiction
(see \cite{CMM} for a similar argument).
Using \eqref{param}, \eqref{beta} and \eqref{est:ab}, it holds, on the interval $[\tstar,\tin]$,
\[
| \dot\sigma-2\beta|\lesssim t^{-\ths} \quad\mbox{and}\quad 
\left|\dot \beta+(1+c)\alpha_c e^{-2c\sigma}\right|\lesssim t^{-1-\ths}.
\]
Set $g=\beta^2-\frac{(1+c)\alpha_c}{2c} e^{-2c\sigma}$, so that by the above estimates and \eqref{eq:betainf} it holds
\begin{equation*}
\dot g = 2\beta\dot \beta +(1+c)\alpha_c \dot \sigma e^{-2c\sigma}=O(t^{-2-\ths})\quad
\mbox{and}\quad g(\tin)=0.
\end{equation*}
By integration on $[t,\tin]$, this yields
\[
\left|\beta^2-\frac{(1+c)\alpha_c}{2c} e^{-2c\sigma}\right|\lesssim t^{-1-\ths}\quad\mbox{and so}\quad
\left|2\beta-\frac{\Omega_c}{c}e^{-c\sigma}\right|\lesssim t^{-\ths}.
\]
Define
\[
\zeta(t)=\frac{e^{c\sigma}}{\Omega_c}\quad\mbox{and}\quad
\xi(t)= \left(\frac{\zeta(t)}{t}-1\right)^{2}.
\]
The previous estimates imply
\begin{equation}\label{est:zeta}
|\dot \zeta (t) -1|\lesssim t^{1-\ths}.
\end{equation}
Assume for the sake of contradiction that for all $\zeta_\sharp\in [-1,1]$, the choice
\[
\zeta(\tin)= \tin+\zeta_\sharp t^{2-\thc}
\]
leads to $\tstar\in (\tzero,\tin]$. By a continuity argument, this means that the bootstrap estimates are reached at $\tstar$.
Since all estimates in \eqref{BS} except the one on $\sigma$, have been strictly improved on $[\tstar,\tin]$, this yields
\begin{equation}\label{saturate}
\bigg| \frac{e^{c\sigma(\tstar)}}{\Omega_c \tstar}-1 \bigg|= \tstar^{1-\thc}.
\end{equation}
Following the argument of \cite{CMM}, we remark that
for any $t\in [\tstar,\tin]$ satisfying~\eqref{saturate},
using \eqref{est:zeta} and $\thc<\ths$, it holds
(taking $T_0$ large enough)
\begin{equation*}
\dot \xi(t) = 2(\dot \zeta(t)-1) (\zeta(t)-t) t^{-2}
-2 (\zeta(t)-t)^2t^{-3}
=- 2t^{1-2\thc} \left( 1+O(t^{\thc-\ths})\right) <0.
\end{equation*}
This transversality condition
implies that  $\tstar$ is a continuous function of 
$\sigma_\infty$ and thus
\[
\Phi : \zeta_\sharp\in [-1,1]\mapsto \tstar^{\thc-2} (\zeta(\tstar)-\tstar)\in \{-1,1\}
\]
is also a continuous function whose image is $\{-1,1\}$, which is contradictory.

To complete the proof of Proposition~\ref{pr:boot}, we observe that from \eqref{param}, $|\dot\gamma_1|+|\dot\gamma_2|\lesssim t^{-\theta}$
holds on the interval $[\tzero,\tin]$.
Integrating and using \eqref{initial}, this gives the uniform estimate $| \gamma_1|+| \gamma_2|\lesssim t^{1-\theta}$ on $[\tzero,\tin]$.
 \end{proof}

\subsection{End of the proof of Theorem~\ref{th:2} by compactness}
We use Proposition~\ref{pr:boot} with $\tin=n$, for any $n\geq \tzero$, to construct a sequence of solutions 
$\big(\begin{smallmatrix} u_n\\v_n\end{smallmatrix}\big) \in \mathcal{C}([\tzero,n],H^1 \times H^1)$
of \eqref{snls} such that, for some $\delta>0$, on $[\tzero, n]$,
\begin{equation}\label{esti:main}
\left\|\begin{pmatrix}
u_n \\ v_n 
\end{pmatrix} - \begin{pmatrix}
e^{\ii c^2 t} Q_c \left(\cdot - \frac{\log t}{c (c+1)} - \frac{\log \Omega_c}{c (c+1)} \right) \\[6pt]
 e^{\ii t} Q \left(\cdot + \frac{\log t}{c+1} - \frac{\log \Omega_c}{c+1} \right)
\end{pmatrix}
\right\|_{H^1 \times H^1} \lesssim t^{-\delta}.
\end{equation}
Now, we adapt from \cite{MMnls} (in the scalar case) and from \cite{IL} (for the vector
case), the following convergence result.
\begin{lemma} There exists $\big(\begin{smallmatrix} u_0\\v_0\end{smallmatrix}\big) \in H^1(\RR) \times H^1(\RR)$ such that
up to a subsequence,
as $n \to \infty$ 
\begin{align*}
&\big(\begin{smallmatrix} u_n\\v_n\end{smallmatrix}\big)(\tzero) \rightharpoonup \big(\begin{smallmatrix} u_0\\v_0\end{smallmatrix}\big)\mbox{ weakly in } H^1(\RR) \times H^1(\RR) \\
&\big(\begin{smallmatrix} u_n\\v_n\end{smallmatrix}\big)(\tzero) \to \big(\begin{smallmatrix} u_0\\v_0\end{smallmatrix}\big) \mbox{ in } H^s(\RR) \times H^s(\RR) \mbox{ for any } 0 \leq s <1.
\end{align*}
\end{lemma}
We consider $\big(\begin{smallmatrix} u\\v\end{smallmatrix}\big)$ the solution of \eqref{snls} corresponding to initial data $\big(\begin{smallmatrix} u_0\\v_0\end{smallmatrix}\big)$ at $t=\tzero$. By $H^1(\RR) \times H^1(\RR)$ boundedness and local well-posedness of Cauchy problem in $H^s(\RR) \times H^s(\RR)$ for any $0 \leq s < 1$ (see \emph{e.g.} \cite{Ca03}), we have the continuous dependence of the solution on the initial data, so for all $t \in [\tzero, + \infty)$,
as $n\to \infty$,
\begin{align*}
&\big(\begin{smallmatrix} u_n\\v_n\end{smallmatrix}\big)(t) \rightharpoonup \big(\begin{smallmatrix} u\\v\end{smallmatrix}\big)(t)\mbox{ in } H^1(\RR) \times H^1(\RR),\\
&\big(\begin{smallmatrix} u_n\\v_n\end{smallmatrix}\big)(t) \to \big(\begin{smallmatrix} u\\v\end{smallmatrix}\big)(t) \mbox{ in } H^s(\RR) \times H^s(\RR) ,\quad 0 \leq s <1.
\end{align*}
Passing to the weak limit as $n \to \infty$ in the uniform estimates \eqref{esti:main}, 
the solution $\big(\begin{smallmatrix} u\\v\end{smallmatrix}\big)$ satisfies Theorem~\ref{th:2}.

\section{Sketch of the proof of Theorem~\ref{th:1}}
\subsection{Approximate solution in the case $c=1$}
In this case, the approximate solution and the solution are symmetric (\emph{i.e.} $u(t,x)=v(t,-x)$) and thus we 
have $\sigma_1=-\sigma_2=\frac{\sigma} 2$, $\beta_1=-\beta_2=\frac{\beta}2$ and $\gamma_1=\gamma_2$.
Using the same notation as in Sections~\ref{S:3} and~\ref{sec:th:2}, we define
(the function $B$ is introduced in Lemma~\ref{le:AB})
\begin{align*}
U&=P+\varphi,\quad P (t,x) = Q(x - \sigma_1 (t)) e^{\ii\Gamma_1(t,x)}, \quad \varphi(t,x) =e^{-\sigma(t)} B(x-\sigma_2(t)) e^{\ii \Gamma_1(t,x)},
\\
V&=R+\psi,\quad R (t,x) = Q(x - \sigma_2 (t)) e^{\ii\Gamma_2(t,x)}, \quad \psi(t,x) =e^{-\sigma(t)} B(x-\sigma_1(t)) e^{\ii \Gamma_2(t,x)}.
\end{align*}
\begin{lemma}
It holds
\begin{equation*}
\EE_{U}= F-\vec{m}_1 \cdot \vec\md_1-\vec{m}_\varphi\cdot\vec\md_\varphi,\\
\end{equation*}
where
\begin{equation*}
F=3|P|^2\varphi+3|\varphi|^2P+|\varphi|^2\varphi
-\omega e^{2 (x -\sigma_1)} |R|^2 P
+\omega(2|R\psi|+|\psi|^2)P,
\end{equation*}
and
\begin{align*}
\vec{m}_1 =
\begin{pmatrix} \dot\gamma_1 +\dot\beta_1 \sigma_1 +\beta_1^2\\ \dot{\sigma}_1-2\beta_1 \\ \dot{\beta}_1 \end{pmatrix},
&\quad \vec\md_1=\begin{pmatrix} \ii \partial_1 P \\P \\ x_1P \end{pmatrix},\\
 \vec{m}_\varphi =
\begin{pmatrix} \dot\gamma_1 +\dot\beta_1 \sigma_2 +\beta_1^2+\ii \dot\sigma\\ \dot{\sigma}_2-2\beta_1 \\ \dot{\beta}_1 \end{pmatrix},
&\quad
\vec\md_\varphi=\begin{pmatrix} \ii \partial_1 \varphi \\ \varphi \\ x_2 \varphi \end{pmatrix}.
\end{align*}
\end{lemma}
We set
\[
a=\frac 12\la F,\partial_1P\ra .
\]
\begin{lemma} It holds
\begin{equation}\label{eq:ab1}
a=\alpha \sigma e^{-2\sigma}+O(e^{-2\sigma})
\quad \mbox{where}\quad \alpha=32 \omega.
\end{equation}
\end{lemma}
\begin{proof}
From the expression of $F$, one has
\begin{align*}
\la F,\partial_1 P\ra 
& =3e^{-\sigma} \int Q^2(x)Q'(x) B(x+\sigma) dx 
+3e^{-2\sigma} \int Q(x)Q'(x) B^2(x+\sigma)dx\\
& \quad +e^{-3\sigma} \int Q'(x)B^3(x+\sigma) dx
-\omega \int e^{2x}Q (x)Q'(x)Q^2(x+\sigma) dx.
\end{align*}
From \eqref{sur:B} and Lemma~\ref{tech}, the second and third terms in the right-hand side are bounded
by $\sigma^3 e^{-4\sigma}$.
The last term is bounded by
\[
\int e^{2x}Q^2(x)Q^2(x+\sigma) dx
=e^{-2\sigma} \int e^{2x}Q^2 (x-\sigma )Q^2(x) dx
\lesssim e^{-2\sigma} \int Q^2 (x-\sigma )dx\lesssim e^{-2\sigma}.
\]
For the first term, using $\Lp Q'=0$ and then \eqref{eq:B}, we compute 
\begin{align*}
3\int Q^2(x)Q'(x) B(x+\sigma) dx
&=\int Q'(x-\sigma) (-B''(x)+B(x)) dx\\
&=\omega \int Q'(x-\sigma) \left[ Q^2(x)B(x)+\kappa e^x Q^2(x)\right] dx
\end{align*}
By Lemma~\ref{tech}, we have $\int |Q'(x-\sigma) Q^2(x)B(x)|dx\lesssim e^{-\sigma}$.

We only have to compute $\int Q'(x-\sigma) e^x Q^2(x)dx$.
First, we see
\[
\int_{x<0} Q'(x-\sigma) e^x Q^2(x) dx
\lesssim e^{-\sigma} \int_{x<0} e^{4x} dx \lesssim e^{-\sigma},
\]
\[
\int_{x>\sigma} |Q'(x-\sigma)| e^x Q^2(x) dx
\lesssim e^{\sigma} \int_{x>\sigma} e^{-2x} dx \lesssim e^{-\sigma}.
\]
Second, using \eqref{asympQ}
\begin{equation*}
Q'(x-\sigma)=\kappa e^{x-\sigma} - e^{2x-2\sigma}Q(x-\sigma),\quad Q^2(x)=\kappa^2e^{-2x}+O(e^{-3x}Q(x)),
\end{equation*}
and thus
\[
\int_0^\sigma Q'(x-\sigma) e^x Q^2(x) dx
=\kappa^3 \sigma e^{-\sigma} + O( e^{-\sigma}).
\]
In conclusion,
$a= \omega\frac{\kappa^4}2 \sigma e^{-2\sigma}
+O(e^{-2 \sigma})
=32 \omega \sigma e^{-2\sigma}$.
\end{proof}

\subsection{Formal discussion for $c=1$}\label{formal1}
The previous computations leads us to
\[
\ddot \sigma =- 4\alpha \sigma e^{-2\sigma} ,\quad 2 \beta = \dot \sigma,
\]
for which the following function is an \emph{approximate solution}
\[
\sigma_0(t)=\log t+\frac 12 \log\log t+\log \Omega,\quad 
2\beta_0(t)=\frac 1t\quad \mbox{where}\quad
\Omega = \sqrt{4\alpha}=8\sqrt{2\omega}.
\]
\subsection{Bootstrap estimates in the case $c=1$}
Fix $\ths$ such that 
$1< \ths<2$. The following bootstrap estimates are used in this case: for $1\ll t\leq \tin$,
\begin{equation*}\left\{\begin{aligned}
&\|\varepsilon\|_{H^1}+\|\eta\|_{H^1}\leq t^{-\ths},\\
& \bigg|\beta-\frac 1{2t}\bigg| \leq t^{-1}(\log t)^{-\frac 14},\\
& \bigg| \frac{e^{\sigma}}{\Omega \sigma^{\frac 12} t}-1 \bigg|\leq (\log t)^{-\frac 12} ,
\end{aligned}\right.\end{equation*}
where  $\sigma_\infty$ is to be chosen satisfying
\begin{equation*}
\bigg| \frac{e^{\sigma_\infty}}{\Omega \sigma_\infty^{\frac 12} \tin}-1 \bigg|\leq (\log \tin)^{-\frac 12}.
\end{equation*}
We refer to \cite{MRlog,NV1} for similar bootstrap estimates.

The rest of the proof is similar to the one of Theorem~\ref{th:2} and we omit it.

\section{Discussion}\label{discuss}
For \eqref{snls}, with any coupling coefficient $0<\omega<1$, we have proved the existence of symmetric $2$-solitary waves (Theorem~\ref{th:1}) and of non-symmetric $2$-solitary waves (Theorem~\ref{th:2}) with logarithmic distance.
Symmetric $2$-solitons with logarithmic distance were already known in the literature for the integrable cases ($\omega=0$ and $\omega=1$) and in the scalar case \eqref{nls}. In contrast, the existence of non-symmetric $2$-solitary waves with logarithmic distance is new. In particular, it does not hold for the integrable case where instead a periodic regime exists.

An interesting remaining open question is whether non-symmetric logarithmic $2$-solitary waves exist for the non-integrable scalar \eqref{nls}. We conjecture that it is indeed the case, as long as $p\neq 3$.
Indeed, the first step of the strategy used in this paper, \emph{i.e.} the computation of an approximate solution involving the main interaction terms, works equally well for \eqref{nls} as for \eqref{snls}.
We expect a logarithmic regime with oscillations.
However,  whereas \eqref{snls} enjoys two $L^2$ conservation laws, the scalar equation~\eqref{nls} enjoys only one, which does not seem sufficient
for the energy method to apply in a context of two solitons with logarithmic distance without symmetry.

A more technical original aspect of this article is the introduction of a refinement of the energy method. 
In previous articles using approximate solutions in the context of error terms of order $t^{-k}$
(\emph{e.g.} in \cite{MRlog,NV1,NV2}), the energy method induces a loss of decay.
Here, the additional correction term $\mathbf{S}$ in Section~\ref{S:ener} allows  an estimate of the remainder
$\big( \begin{smallmatrix} \varepsilon\\ \eta\end{smallmatrix}\big)$ directly related to the size of  the error term 
$\big( \begin{smallmatrix} \EE_U\\ \EE_V\end{smallmatrix}\big)$.
We believe that this general observation will be useful elsewhere.


\begin{thebibliography}{10}

\bibitem{APT}
M. J. Ablowitz, B. Prinari, and A. D. Trubatch.
\emph{Discrete and continuous nonlinear Schrodinger systems.}
London Mathematical Society Lecture Note Series (302), Cambridge University Press, Cambridge, 2004.

\bibitem{Berge}
L.~Bergé. Wave collapse in physics: principles and applications to light and plasma waves. \emph{ Phys. Rep.}, 303 5-6 (1998), 259-370.

\bibitem{Ca03}
T.~Cazenave. \emph{Semilinear {S}chr\"odinger equations.} {Courant Lecture Notes in Mathematics}, New York University, New York, 2003.

\bibitem{CM}
V. Combet and Y. Martel. Construction of multi-bubble solutions for the critical gKdV equation. 
\emph{ SIAM J. Math. Anal.}, 50(4) (2018), 3715-3790.

\bibitem{CMM}
R. Côte, Y. Martel and F. Merle. Construction of multi-soliton solutions for the L2-supercritical gKdV and NLS equations.
\emph{ Rev. Mat. Iberoam.} 27 (2011), no. 1, 273-302.

\bibitem{DCW}
F. Delebecque, S. Le Coz and R. M. Weishäupl.
Multi-speed solitary waves of nonlinear Schrödinger systems: theoretical and numerical analysis. 
\emph{Commun. Math. Sci.} 14 (2016), no. 6, 1599-1624.

\bibitem{FT}
L. D. Faddeev and L. A. Takhtajan. 
\emph{Hamiltonian methods in the theory of solitons.}
{Springer-Verlag}, 2007.

\bibitem{GO}
K. A. Gorshkov and L.A. Ostrovsky.
Interactions of solitons in non-integrable systems: direct perturbation method and applications.
\emph{ Physica 3D}, 1\&2 (1981), 428-438.

\bibitem{GSS}
M.~Grillakis, J.~Shatah and W.~A.~Strauss.
Stability theory of solitary waves in the presence of symmetry.
\emph{ J. Funct. Anal.}, 197 (1987), 74-160.

\bibitem{GV}
J. Ginibre and G. Velo. On a class of nonlinear Schrödinger equations. I. The Cauchy problem, general case.
\emph{ J. Funct. Anal.} 32 (1979), 1-32.

\bibitem{IL}
I.~Ianni and S.~Le Coz.
Multi-speed solitary wave solutions for nonlinear Schr\"odinger system
\emph{ J. Lond. Math. Soc.} 89(2) (2014), 623-639.

\bibitem{J}
J. Jendrej. Dynamics of strongly interacting unstable two-solitons for generalized Korteweg-de Vries
equations.
Preprint 	arXiv:1802.06294

\bibitem{KS}
V. I. Karpman, and V.V. Solov'ev. A perturbational approach to the two-soliton system.
\emph{ Physica 3D}, 1\&2 (1981), 487-502.

\bibitem{KMR}
J.~Krieger, Y.~Martel and P.~Rapha\"el.
 Two-soliton solutions to the three-dimensional gravitational Hartree equation.
\emph{ Comm. Pure Appl. Math.} 62 (2009), no. 11, 1501-1550.

\bibitem{Man} 
S. V. Manakov.
On the theory of two-dimensional stationary self-focusing of electromagnetic waves.
\emph{ Journal of Experimental and Theoretical Physics}, 38 (1974), 248-253.

\bibitem{Martel1} 
Y. Martel. Asymptotic N-soliton-like solutions of the subcritical and critical generalized Korteweg-de
Vries equations.
\emph{ Amer. J. Math}, 127 (2005), no. 5, 1103-1140.

\bibitem{MMnls} 
Y.~Martel and F.~Merle.
Multi-solitary waves for nonlinear Schr\"odinger equations.
\emph{ Annales de l'IHP (C) Non Linear Analysis}, 23 (2006), 849-864.

\bibitem{MMT1}
Y. Martel, F. Merle and T.-P. Tsai.
Stability and asymptotic stability in the energy space of the sum of $N$ solitons
for the subcritical gKdV equations.
\emph{Commun. Math. Phys.}, 231 (2002) 347-373.

\bibitem{MMT2} 
Y. Martel, F. Merle and T.-P. Tsai.
Stability in $H^1$ of the sum of $K$ solitary waves for some nonlinear Schr\"odinger equations.
\emph{ Duke Math. J.} {133} (2006), 405-466.

\bibitem{MRlog}
Y.~Martel and P.~Rapha\"el.
Strongly interacting blow up bubbles for the mass critical NLS.
\emph{ Annales scientifiques de l’École normale supérieure,} 51 (2018), 701-737. 

\bibitem{Mmulti}
F. Merle. Construction of solutions with exactly $k$ blow-up points for the Schrödinger equation with
critical nonlinearity.
\emph{ Comm. Math. Phys.}, 129 (1990), no. 2, 223-240.
 
\bibitem{NV1}
T. V. Nguyen. Existence of multi-solitary waves with logarithmic relative distances for the NLS equations.
\emph{ C. R. Acad. Sci. Paris}, Ser. I 357 (2019) 13–58.

\bibitem{NV2}
T. V. Nguyen. Strongly interacting multi-solitons with logarithmic relative distance for the gKdV
equation.
\emph{ Nonlinearity}, 30(12):4614, 2017.

\bibitem{Olm}
E. Olmedilla. Multiple pole solutions of the nonlinear Schr\"odinger equation.
\emph{ Physica D.}, 25 (1987), 330-346.

\bibitem{PZur}
P.~Rapha\"el.
 Stability and blow up for the nonlinear Schr\"odinger equation.
 {\em Lecture notes for the Clay summer school on evolution equations}, ETH, Zurich (2008).

\bibitem{RaSz11}
P.~Rapha\"el and J.~Szeftel.
Existence and uniqueness of minimal blow-up solutions to an
inhomogeneous mass critical {NLS}.
\emph{ J. Amer. Math. Soc.}, 24(2):471-546, 2011. 

\bibitem{T}
E.C. Titchmarsh. \emph{Eigenfunction expansions associated with second-order differential equations.}
Clarendon Press, Oxford, 1946.

\bibitem{We83}
M.~I.~Weinstein.
Nonlinear {S}chr\"odinger equations and sharp interpolation estimates.
\emph{ Comm. Math. Phys.}, 87(4):567-576, 1982/83.
 
\bibitem{We85}
M.~I.~Weinstein.
Modulational stability of ground states of nonlinear Schr\"odinger equations.
\emph{ SIAM J. Math. Anal.}, 16 (1985), 472-491.

\bibitem{We86}
M.~I. Weinstein.
Lyapunov stability of ground states of nonlinear dispersive evolution equations.
\emph{ Comm. Pure Appl. Math.}, 39 (1986), 51-68.
 
\bibitem{Ya1}
J. Yang. Suppression of Manakov-soliton interference in optical fibers.
\emph{ Rev. E.}, 65, 036606 (2002).
 
\bibitem{Ya}
J.~Yang.
\emph{Nonlinear waves in integrable and non-integrable systems}.
SIAM Philadelphia 2010.

\bibitem{ZS}
T.~Zakharov and A.B.~Shabat.
Exact theory of two-dimensional self-focusing and one-dimensional self-modulation of waves in nonlinear media.
\emph{ Sov. Phys. JETP} 34 (1972), 62-69.
\end{thebibliography}
\end{document}